\documentclass[12pt]{amsart}
\usepackage{amsmath,latexsym,amsfonts,amssymb,mathrsfs,ulem}
\usepackage{graphicx,verbatim}
\usepackage{a4wide}
\usepackage{mathtools} 
\usepackage{stackengine} 
\usepackage[colorlinks,linkcolor={blue},
citecolor={blue},urlcolor={red},]{hyperref}
\usepackage{hyperref}
\usepackage{color}
\allowdisplaybreaks

\theoremstyle{plain}
\newtheorem{theorem}{Theorem}[section]
\theoremstyle{remark}
\newtheorem{remark}[theorem]{Remark}

\theoremstyle{plain}
\newtheorem{corollary}[theorem]{Corollary}
\newtheorem{lemma}[theorem]{Lemma}
\newtheorem{proposition}[theorem]{Proposition}
\newtheorem{definition}[theorem]{Definition}

\newtheorem{assumption}{Assumption}

\numberwithin{equation}{section}

\newcommand{\rA}{\mathrm{ A}}
\newcommand{\rH}{\mathrm{ H}}

\newcommand{\rV}{\mathrm{ V}}

\newcommand{\rU}{\mathrm{ U}}
\newcommand{\hH}{\mathbb{H}}
\newcommand{\lL}{\mathbb{L}}
\newcommand{\LS}{{\lL^2(Q)}}
\newcommand{\LQeps}{{\lL^2(Q_\eps)}}
\newcommand{\bbS}{Q}
\newcommand{\nablaS}{\nabla'}
\newcommand{\DDelta}{\Delta'}
\newcommand{\ddivS}{\ddiv'}

\newcommand{\ddiv}{{\rm div\,}}

\newcommand{\tu}{\widetilde{u}}
\newcommand{\rv}{\mathrm{v}}
\newcommand{\rw}{\mathrm{w}}

\newcommand{\hu}{\what{u}}

\newcommand{\tge}{\widetilde{g}_\eps}
\newcommand{\tfe}{\widetilde{f}_\eps}

\def\N{\mathbb{N}}

\newcommand{\R}{\mathbb{R}}
\newcommand{\Rp}{\mathbb{R}_+}

\newcommand{\tale}{\wtd{\alpha}_\eps}
\newcommand{\tble}{\wtd{\beta}_\eps}
\newcommand{\ale}{{\alpha}_\eps}

\newcommand{\hae}{\what{\alpha}_\eps}

\newcommand{\hbe}{\what{\beta}_\eps}

\newcommand{\hp}{\what{\mathbb{P}}}
\newcommand{\homega}{\what{\Omega}}
\newcommand{\hE}{\what{\E}}
\newcommand{\bbF}{\mathbb{F}}
\newcommand{\bbP}{\mathbb{P}}
\newcommand{\E}{\mathbb{E}}
\newcommand{\bbE}{\mathbb{E}}
\newcommand{\tWe}{\wtd{W}_\eps}

\newcommand{\be}{\begin{equation}}
\newcommand{\ee}{\end{equation}}
\newcommand{\ba}{\begin{array}}
\newcommand{\ea}{\end{array}}
\newcommand{\beas}{\begin{eqnarray*}}
\newcommand{\eeas}{\end{eqnarray*}}
\newcommand{\bea}{\begin{eqnarray}}
\newcommand{\eea}{\end{eqnarray}}
\newcommand{\nn}{\nonumber}
\newcommand{\lb}{\label}
\newcommand{\calA}{\mathcal{A}}
\newcommand{\calB}{\mathcal{B}}

\newcommand{\calF}{\mathcal{F}}
\newcommand{\calK}{\mathcal{K}}

\newcommand{\calT}{\mathcal{T}}
\newcommand{\calV}{\mathcal{V}}

\newcommand{\calZ}{\mathcal{Z}}

\newcommand{\test}{{ R_\eps\varphi}}
\newcommand{\tMe}{\widetilde{M}_\eps}

\newcommand{\oMe}{\ocirc{M}_\eps}
\newcommand{\tNe}{\widetilde{N}_\eps}
\newcommand{\tue}{\widetilde{u}_\eps}

\newcommand{\eps}{\varepsilon}
\newcommand{\wtd}{\widetilde}
\newcommand{\what}{\widehat}
\newcommand{\bx}{{\bf x}}
\newcommand{\by}{{\bf y}}
\newcommand{\inteps}{\int_0^\eps}
\newcommand\ocirc[1]{\ensurestackMath{\stackon[1pt]{#1}{\mkern2mu\circ}}}
\definecolor{darkblue}{rgb}{0.1,0.1,0.9}

\newcommand\delc[1]{}
\newcommand\dela[1]{}

\title{Stochastic Navier--Stokes equations on a 3D thin domain}
\author{Z.~Brze\'{z}niak}
\address{Department of Mathematics, University of York, Heslington, York, YO10 5DD, UK}
\email{zdzislaw.brzezniak@york.ac.uk}
\author{G.~Dhariwal}
\address{Institute of Analysis and Scientific Computing, Vienna University of Technology, Vienna, Austria}
\email{gaurav.dhariwal@tuwien.ac.at}
\author{Q.~T.~Le Gia}
\address{School of Mathematics and Statistics, University of New South Wales, Sydney, NSW 2052, Australia}
\email{qlegia@unsw.edu.au}
\thanks{The research of all three authors is partially supported by Australian Research Council Discover Project grant DP180100506. Zdzis\l aw Brze{\'z}niak has been supported by the Leverhulme project grant ref no RPG-2012-514 and by Australian Research Council Discover Project grant DP160101755. The research of Gaurav Dhariwal was supported by Department of Mathematics, University of York and is partially supported by the Austrian Science Fund (FWF) grants P30000, W1245, and F65.
}
\dedicatory{This article is dedicated to the memory of Igor D. Chueshov, a pioneer in the fields of non-linear PDE, dissipative systems, monotone stochastic dynamical systems, to state a few.}
\begin{document}
\date{\today}

\begin{abstract} 
    Stochastic Navier--Stokes equations in a thin three-dimensional domain are considered, driven by additive noise. The convergence of martingale solution of the stochastic Navier--Stokes equations in a thin three-dimensional domain to the unique martingale solution of the 2D stochastic Navier--Stokes equations, as the thickness of the film vanishes, is established. Hence, we justify the approximation of 3D Navier--Stokes equations driven by random forcing by its corresponding two-dimensional setting in applications.
\end{abstract}
\maketitle

\section{Introduction}

For various motivations, partial differential equations in thin domains has been studied extensively in last few decades;
e.g. Babin and Vishik \cite{[BV83]}, Ciarlet \cite{[Ciarlet90]}, Ghidaglia and Temam \cite{[GT91]}, Marsden \textit{et.al.} \cite{[MRR95]}, Kaizu and Saito \cite{[KS07]} and references there in. The study of the Navier--Stokes equations (NSE) on thin domains originates in a series of papers by Hale and Raugel \cite{[HR92a]}--\cite{[HR92c]} concerning the reaction-diffusion and damped wave equations on thin domains. Raugel and Sell \cite{[RS93],[RS94]} proved the global existence of strong solutions to NSE on thin domains for large initial data and forcing terms, in the case of purely periodic and periodic-Dirichlet boundary conditions. Later,
by applying a contraction principle argument and carefully analysing the dependence of the solution on the first eigenvalue of the corresponding Laplace operator, Arvin \cite{[Arvin96]} showed global existence of strong solutions of the Navier--Stokes equations on thin three-dimensional domains for large data. Temam and Ziane \cite{[TZ96]} generalised the results of \cite{[RS93], [RS94]} to other boundary conditions. Moise \textit{et.al.} \cite{[MTZ97]} proved global existence of strong solutions for initial data larger than in \cite{[RS94]}. Iftimie \cite{[Iftime99]} showed the existence and uniqueness of solutions for less regular initial data which was further improved by Iftimie and Raugel \cite{[IR01]} by reducing the regularity and increasing the size of initial data and forcing.

More recently, stochastic 3D Navier--Stokes equations in a thin three dimensional
domain $\mathbb{O}_\eps = \mathbb{T}^2 \times (0,\eps)$,
where $\mathbb{T}^2$ is a two-dimensional
torus, was investigated by Chueshov and Kuksin in \cite{[CK08]} in which
the authors considered the 3D NSE in $\mathbb{O}_\eps$ perturbed by a random kick-force
and proved that if the force is not too big and is genuinely random then the equation has
a unique stationary measure.

The main objective of our paper is to establish the convergence of the martingale
solution $\tue$ (more specifically the averages in thin direction $\oMe \tue$, see \eqref{eqn-M_eps-0} for the definition of $\oMe$) of the stochastic Navier--Stokes equations (SNSE) described in a thin domain $Q_\eps$ \cite{[FG95],[BM13],[BS20]}, to the martingale solution of the 2D stochastic Navier--Stokes equations on a bounded domain $Q \subset \R^2$ \cite{[BM13]}, as thickness $\eps$ of the thin domain converges to zero, Theorem~\ref{thm:main_thm}. 

{After having finished this article, we came across the paper \cite{[CK08a]} by Chueshov and Kuksin, in  which the authors proved the existence and the uniqueness of a solution $u^\alpha_\eps$ to the so called Leray-$\alpha$ approximation; see \cite{[Leray34]}, of the stochastic 3D Navier--Stokes equations \eqref{eq:leray_NSE} and in Proposition~12 \cite{[CK08a]}, a result similar to ours is stated, see Remark \ref{rem-Kuksin} for more detailed discussion.}

We study the stochastic Navier--Stokes equations (SNSE) for incompressible fluid
\begin{align}
 d \tue - [ \nu \Delta \tue - (\tue \cdot \nabla) \tue - \nabla \widetilde{p}_\eps]dt = \tfe dt +  \widetilde{G}_\eps\, d\tWe(t) &\quad \text{ in } Q_\eps \times (0,T), \label{eqn-snse1}\\
\ddiv \tue = 0 &\quad \text{ in } Q_\eps \times (0,T),\label{eqn-snse2}
\end{align}
in thin domain
\begin{equation}\label{eqn-Q-eps}
Q_\eps := Q \times (0,\eps), \text{ where } 0< \eps < 1/2,
\end{equation}
and $Q \subset \R^2$ is a smooth bounded domain with smooth boundary $\partial Q$ along with boundary conditions
\begin{align}
\wtd{u}_{\eps,3} = 0, \quad \partial_3 \wtd{u}_{\eps,j}=0, \;\; j=1,2
&\quad\text{ on } \Gamma^\eps_h \times (0,T), \label{eqn-snse3} \\
\wtd{u}_\eps = 0 &\quad \text{ on  } \Gamma^\eps_{\ell} \times (0,T),\label{eqn-snse4} \\
\wtd{u}_\eps(\cdot,0) = \wtd{u}^\eps_0 &\quad\text{ in } Q_\eps. \label{eqn-snse5}
\end{align}
In the above, $\tue=(\wtd{u}_{\eps,1}, \wtd{u}_{\eps,2}, \wtd{u}_{\eps,3})$ is the fluid velocity field, $p$ is the pressure, $\nu>0$ is a (fixed) kinematic viscosity, $\tu_0^\eps$ is a divergence free
vector field on $Q_\eps$ and
\begin{equation}
\label{eqn-Gamma}
\Gamma^\eps_\ell = \partial Q \times (0,\eps), \quad
\Gamma^\eps_h = \overline{Q} \times \{0, \eps \},
\end{equation}
$\partial Q_\eps = \Gamma^\eps_\ell \cup \Gamma^\eps_h$.
We also put
\[
\Gamma^{\eps,+}_h = \overline{Q} \times \{\eps\}\; \mbox{ and }\;\Gamma^{\eps,-}_h = \overline{Q} \times \{0\}.
\]
$\tWe(t)$, $t \ge 0$ is an $\R^N$-valued Wiener process in some probability space $\left(\Omega, \mathcal{F}, \mathbb{F}, \mathbb{P}\right)$ to be defined precisely later.

The main result of this paper is Theorem~\ref{thm:main_thm} where we prove the convergence of the averages
in the vertical direction of the martingale solution (see Definition~\ref{defn5.2}) of the 3D stochastic equations
\eqref{eqn-snse1}--\eqref{eqn-snse5}, as the thickness $\eps \rightarrow 0$, to a martingale solution $u$ (see Definition~\ref{defn_mart_SNSE_sphere}) of the following 2D stochastic Navier--Stokes equations in $Q$:
\begin{align}
du - \left[\nu \Delta' u - (u \cdot \nabla') u
 -  \nabla'p\right]dt = fdt + G\,dW(t)
&\quad\text{ in } Q \times (0,T), \label{eqn-snse11}\\
\ddiv' u = 0
&\quad \text{ in } Q \times (0,T),\label{eqn-snse12} \\
u(x',0) = u_0
&\quad\text{ in } Q,\label{eqn-snse13}\\
u = 0
&\quad\text{ on } \partial Q \times (0,T) \label{eqn-snse14}
\end{align}
where $u=(u_1,u_2)$, $\Delta'$, $\nabla'$ and $\ddiv'$ are the Laplacian, gradient and divergence operators respectively defined for a $\R^2$-valued function. Assumptions on initial data and external forcing will be stated later.

An interesting observation here is that, for $\eps > 0$, existence of a martingale solution for the 3D SNSE \eqref{eqn-snse1}--\eqref{eqn-snse5} is known, see \cite{[FG95],[BM13]}, but the uniqueness of such a martingale solution is not known (as is the case for deterministic 3D NSE). Whereas,  for the 2D NSE in a bounded domain \eqref{eqn-snse11}--\eqref{eqn-snse14}, there exists a unique martingale solution, see \cite{[BP01], [CK08]} and references therein. Thus, irrespective of which martingale solution $\tue$ of \eqref{eqn-snse1}--\eqref{eqn-snse5} we choose in Theorem~\ref{thm:main_thm}, the limiting object $\hu$ will be the same, the unique martingale solution to \eqref{eqn-snse11}--\eqref{eqn-snse14}.

The article is organised as follows. We introduce necessary functional spaces, Stokes operator and bilinear map $B$ corresponding to the nonlinearity in \S\,\ref{section-prelim}. In \S\,\ref{sec:Average}, we define some averaging operators which are essential tools required in the further analysis and give their properties. Stochastic Navier--Stokes equations on a thin domain $Q_\eps$ are introduced in \S\,\ref{sec:snse} and a priori estimates for the vertically averaged velocity are obtained which are later used to prove the convergence of the average of a martingale solution $\tue$, of stochastic NSE on a thin domain (see \eqref{eq:5.1}--\eqref{eq:5.4}), to a martingale solution of the 2D stochastic NSE on $Q$ (see \eqref{eq:5.6}--\eqref{eq:5.9}) with vanishing thickness, see Theorem~\ref{thm:main_thm}.

\textbf{Acknowledgement.} The authors would like to thank Sergei Kuksin for discussion related to his paper \cite{[CK08]}.

\section{Preliminaries}\label{section-prelim}

A point $\bx \in Q_\eps$ will be represented in the Cartesian coordinates $\bx = (x_1,x_2,x_3)$ or by $\bx = (\bx', x_3)$ where $\bx' = (x_1,x_2) \in Q$ and $x_3 \in (0, \eps)$.

\subsection{Function spaces}\label{subsection-spaces}
For $p \in [1, \infty)$, by $L^p(Q_\eps)$, we denote the Banach space of (equivalence-classes of) Lebesgue measurable $\R$-valued $p$th power integrable functions on $Q_\eps$. The $\R^3$-valued $p$th power integrable vector fields will be denoted by $\mathbb{L}^p(Q_\eps)$. The norm in $\mathbb{L}^p(Q_\eps)$ is given by
\[\|u\|_{\mathbb{L}^p(Q_\eps)} := \left(\int_{Q_\eps}|u(\by)|^p\,d\by\right)^{1/p}, \qquad u \in \mathbb{L}^p(Q_\eps).\]
If $p=2$, then $\mathbb{L}^2(Q_\eps)$ is a Hilbert space with the inner product given by
\[
\left(u, \rv \right)_{\LQeps} := \int_{Q_\eps} u(\by) \cdot \rv(\by)\,d\by, \qquad u, \rv \in \LQeps.
\]
The Sobolev space $\mathbb{W}^{k,p}(Q_\eps; \R^3)$ ($k$ is integer, $1 \le p \le \infty$) is the space of $\R^3$-valued functions in $\lL^p(Q_\eps)$ with weak derivatives of ordr less than or equal to $k$ in $\lL^p(Q_\eps)$. This is a Banach space with the norm
\[
\|u\|_{\mathbb{W}^{k,p}(Q_\eps)} := \left(\sum_{[j]\le m} \|D^j u\|^p_{L^p(Q_\eps)}\right)^{1/p}.
\]
When $p = 2$, $\hH^k(Q_\eps)$ is a Hilbert space with the scalar product
\[
(u, \rv)_{\hH^{k}(Q_\eps)} = \sum_{[j] \le k} (D^ju, D^j \rv), \qquad u, \rv \in \hH^k(Q_\eps).
\]
In particular for $k = 1$, $\mathbb{H}^1(Q_\eps) = \mathbb{W}^{1,2}(Q_\eps)$ is a Hilbert space with the inner product given by
\[ \left( u, \rv \right)_{\mathbb{H}^1(Q_\eps)} := \left(u, \rv \right)_\LQeps + \left(\nabla u, \nabla \rv \right)_\LQeps, \qquad u, \rv \in \mathbb{H}^1(Q_\eps),\]
where
\[\left(\nabla u, \nabla \rv \right)_\LQeps = \sum_{i=1}^3\int_{Q_\eps} D_i u(\by) \cdot D_i \rv(\by)\,d\by.\]
The Lebesgue and Sobolev spaces on $\bbS$ for $R^2$-valued vector fields will be denoted by $\mathbb{L}^p(\bbS)$ and $\mathbb{W}^{m,q}(\bbS)$ respectively for $1 \le p, q \le \infty$ and $m \in \mathbb{Z}$. In particular, we will write $\mathbb{H}^1(\bbS)$ for $\mathbb{W}^{1,2}(\bbS)$.

Next, we recall the ususal spaces appearing in the analysis of NSE:
\[\calV = \{ \varphi \in C^\infty(Q; \R^2) : \ddiv' \varphi = 0  \},\]
\begin{equation}\label{eqn-H}
 \rH = \text{closure of } \calV \text{ in } \lL^2(Q),
\end{equation}
and
\begin{equation}\label{eqn-V}
 \rV = \mbox{closure of } \calV \mbox{ in } \hH^{1}_0,
\end{equation}
where $\mathbb{W}^{k,p}_0(Q)$ ($\hH^k_0(Q)$ when $p=2$) is the closure of
$C_0^\infty(Q; \R^2)$\footnote{the space of $C^\infty(Q;\R^2)$ functions with compact support in $Q$} in $\mathbb{W}^{k,p}(Q)$. The functional spaces $\rH$ and $\rV$ can be further characterised \cite{[Temam00]} as
\begin{equation}\label{eqn-H2}
 \rH =  \{ u\in \lL^2(Q; \R^2) : \ddiv' u = 0 \mbox{ in } Q \mbox{ and } u \cdot \vec{n} = 0 \mbox{ on } \partial Q  \},
\end{equation}
\begin{equation}\label{eqn-V2}
 \rV =  \{ u\in \hH^{1,2}(Q; \R^2): \ddiv' u = 0 \mbox{ in } Q \mbox{ and } u = 0 \mbox{ on } \partial Q  \},
\end{equation}
where $\vec{n}$ is the external unit normal vector field to $\partial Q$. On $\rH$, we consider the inner product and the norm inherited from $\lL^2(Q)$ and denote then by $(\cdot, \cdot)_\rH$ and $\|\cdot\|_\rH$ respectively, i.e.
\[
\left(u, \rv\right)_\rH := \left(u, \rv\right)_\LS, \qquad \qquad \|u\|_\rH := \|u\|_\LS, \quad u, \rv \in \rH.
\]
The norm on $\rV$ is equivalent to the norm:
\[\|u\| = \sum_{i=1}^2 |D_iu|^2, \qquad u \in \rV \]
and will be denoted by $\|\cdot\|_\rV$.

Similarly, see \cite{[TZ96]}, we define the spaces $\rH_\eps$ and $\rV_\eps$ for the vector fields defined on $Q_\eps$, by
\begin{equation}\label{eqn-H_eps}
 \rH_\eps =  \{ u\in \lL^2(Q_\eps;\R^3) : \ddiv u = 0 \mbox{ in } Q_\eps \mbox{ and } u \cdot \vec{n}_\eps = 0 \mbox{ on } \partial Q_\eps  \},
\end{equation}
\begin{equation}\label{eqn-V_eps}
  \rV_\eps=  \{ u\in \hH^1(Q_\eps;\R^3)\cap \rH_\eps : u  = 0 \mbox{ on } \Gamma_l^\eps\},
\end{equation}
where $\vec{n}_\eps$ is the external unit normal vector field to $\partial Q_\eps$,
e.g. $\vec{n}_\eps(\bx',x_3) = (\vec{n}(\bx'),0)$ for $(\bx',y) \in \Gamma^\eps_\ell$.
Note also that $\vec{n}_\eps = (0,0,\pm 1)$ on $\Gamma^{\eps,\pm}_h$.

If $u \in \rH_\eps$ then $u_3 = 0$ on $\Gamma_h^\eps$.
Thus, $\rH_\eps$ can be further characterised as
\begin{equation}\label{eqn-H-u_3}
 \rH_\eps =  \{ u\in \lL^2(Q_\eps;\R^3) : \ddiv u = 0 \mbox{ and } u \cdot \vec{n}_\eps = 0 \mbox{ on }  \Gamma_l^\eps \mbox{ and } u_3=0 \mbox{ on }\Gamma_h^\eps   \}.
\end{equation}

As in the case of spaces $\rH$ and $\rV$, the norms and inner products on the functional spaces $\rH_\eps$, $\rV_\eps$ are inherited from $\LQeps$ and $\hH^1_0(Q_\eps)$ and will be denoted by $\|\cdot\|_{\rH_\eps}$, $(\cdot, \cdot)_{\rH_\eps}$ and $\|\cdot\|_{\rV_\eps}$, $(\cdot, \cdot)_{\rV_\eps}$ respectively.

For later use, let us introduce the following function space:
\begin{equation}
    \label{space-U}
\rU := \text{closure of } \calV \text{ in } \hH^2(Q).
\end{equation}

\subsection{The Stokes operator}\label{subsection-Stokes}
The Stokes operator $\rA: D(\rA) \rightarrow \rH$ for $D(\rA) = \rH \cap H_0^2(Q,\R^2)$ is given by
\begin{equation}
\label{eqn-A}
  \rA u = - \pi (\Delta^\prime u),
\end{equation}
where $\Delta^\prime = \partial^2_1 + \partial^2_2$ is the Laplacian on $\R^2$ and $\pi$ is the Leray-Helmholtz projection operator from $L^2(Q; \R^2)$ onto $\rH$. The operator $\rA$ is associated
with a coercive bilinear form $a:\rV \times \rV \rightarrow \R$,
\[
  a(u,\rv) = \int_Q \nabla' u \cdot \nabla' \rv\, dx'
\]
and a linear isomorphism $\calA:\rV \rightarrow \rV^\ast$ satisfying
\[
 {}_{\rV^\ast}\langle\calA u , \rv\rangle_\rV = a(u,\rv), \quad u, \rv \in \rV,
\]
where $\rV^\ast$ is the topological dual of $\rV$.
It is known, see e.g. \cite{[LM72]} that
\begin{equation}
\label{eqn-A2}
D(\rA) = \{ u \in \rV: \calA u \in \rH\} \text{ and }
\rA u = \calA u, \quad u \in D(\rA),
\end{equation}
and $\rA$ is a self-adjoint, non-negative operator with a compact inverse. Moreover, $D(\rA)$ can be further characterised:
\begin{equation}
    \label{eqn-A3}
D(\rA) = \{ u \in \hH^{2,2}(Q) \colon  \ddiv^\prime u = 0 \mbox{ in } Q,\,
   u\cdot\vec{n} = 0 \text{ and } u = 0 \text{ on } \partial Q\}.
\end{equation}

Analogously, we consider a coercive symmetric bilinear form
$a_{\eps}: \rV_{\eps} \times \rV_{\eps} \rightarrow \R$
\[
a_\eps(u,\rv) = \int_{Q_\eps} \nabla u \cdot \nabla \rv\,dx, \quad u,\rv \in \rV_\eps
\]
and a linear isomorphism $\calA_\eps: \rV_\eps \rightarrow \rV^\ast_\eps$ satisfying
\[
 {}_{\rV^\ast_\eps}\langle\calA_\eps u, \rv\rangle_{\rV_\eps} = a_\eps(u, \rv), \quad u, \rv \in \rV_\eps,
\]
$\rV^\ast_\eps$ is the topological dual of $\rV_\eps$.
Then we define
\begin{equation}
\label{eqn-Aeps}
D(\rA_\eps) = \{ u \in \rV_\eps: \calA u \in \rH_\eps \} \text{ and }
\rA_\eps u = \calA_\eps u, \quad u \in D(\rA_\eps).
\end{equation}

It is known that $D(\rA_\eps)$ can be also characterised \cite{[TZ96]} as follows
\begin{equation}
\label{eqn-Aeps2}
D(\rA_\eps) = \{ u \in \hH^{2,2}(Q_\eps) \colon  \ddiv u = 0 \mbox{ in } Q_\eps,\,
   u = 0 \text{ on } \Gamma_\ell^\eps ,\, u_3, \partial_3 u_j = 0 \text{ on } \Gamma^\eps_h, \, j=1,2 \}.
\end{equation}
Moreover, by \cite{[Triebel78]}
\begin{equation}\label{eqn-fractional domains}
\begin{split}
\rV=[\rH, D(\rA)]_{\frac12} &\;\mbox{ and }\;  \rV_\eps = [\rH_\eps, D(\rA_\eps)]_{\frac12} \\
\rV = D(\rA^{1/2}) &\;\mbox{ and }\;  \rV_\eps = D(\rA^{1/2}_\eps).
\end{split}
\end{equation}
Note that for $\rH_\eps \ni u_\eps = (u^\prime, u_{\eps,3})$, we have the condition $ u_\eps \cdot \vec{n}_\eps  = u' \cdot \vec{n} = 0$
on $\Gamma^\eps_\ell$ which is weaker than the condition $u_\eps = 0$ on $\Gamma_\ell^\eps$ for $u_\eps \in \rV_\eps$.

It can be proved that $A_\eps$ is self-adjoint in $\rH_\eps$ and from \eqref{eqn-Aeps} and the definition of the bilinear form $a_\eps$ we conclude
\begin{equation}\label{def:Aeps}
 {}_{\rV^\ast_\eps}\langle A_\eps u, \rv\rangle_{\rV_\eps} = \int_{Q_\eps} (\nabla u, \nabla \rv)_{\lL^2(Q_\eps)}\,dx,
\quad u, \rv \in \rV_\eps.
\end{equation}


\subsection{The nonlinear term}
\label{sec:nonlinear}
Let $b_\eps$ be the continuous trilinear from on $\rV_\eps$ defined by$\colon$
\begin{equation}
\label{eqn-trilin-eps}
b_\eps(u, \rv, \rw) = \int_{Q_\eps}\left(u \cdot \nabla \right) \rv \cdot \rw \,d\bx,\qquad u,\rv,\rw \in \rV_\eps.
\end{equation}
We denote by $B_\eps$ the bilinear mapping from $\rV_\eps\times\rV_\eps$ to $\rV_\eps^\ast$ by
\begin{equation}
\label{eqn-nonlin-eps}
{}_{\rV^\ast_\eps}\left\langle B_\eps(u,\rv), \rw \right\rangle_{\rV_\eps} = b_\eps(u,\rv,\rw), \qquad u,\rv, \rw \in \rV_\eps,
\end{equation}
and we set
\[
B_\eps(u) = B_\eps(u,u).
\]
Let us also recall the following properties of the form $b_\eps$, which directly follows from the definition of $b_\eps$ $\colon$
\begin{equation}
\label{eqn-tri-eps-prop1}
b_\eps(u,\rv,\rw) = - b_\eps(u,\rw,\rv),\qquad u,\rv, \rw \in \rV_\eps.
\end{equation}
In particular,
\begin{equation}
\label{eqn-tri-eps-prop2}
{}_{\rV^\ast_\eps}\left\langle B_\eps(u,\rv), \rv \right\rangle_{\rV_\eps} = b_\eps(u, \rv, \rv) = 0, \qquad u,\rv \in \rV_\eps.
\end{equation}

Analogously, we can define trilinear form $b$ on $\rV$ and bilnear map $B \colon \rV \times \rV \to \rV^\ast$ corresponding to the nonlinearity appearing in 2D NSE, i.e.
\begin{align}
\label{eqn-trilin}
b(u, \rv, \rw) = \int_{Q}\left(u \cdot \nabla \right) \rv \cdot \rw \,d\bx,\qquad u,\rv,\rw \in \rV, \\
\label{eqn-nonlin}
{}_{\rV^\ast}\left\langle B(u,\rv), \rw \right\rangle_{\rV} = b(u,\rv,\rw), \qquad u,\rv, \rw \in \rV.
\end{align}
The properties analogous to \eqref{eqn-tri-eps-prop1} and \eqref{eqn-tri-eps-prop2} hold for maps $b$ and $B$ too.

\section{Averaging operators and their properties}\label{sec:Average}
\subsection{Some definitions}
In this section we recall the averaging operators which were first introduced by Raugel and Sell \cite{[RS93], [RS94]} to study the well-posedness of Navier--Stokes equations on thin 3D domains and existence of global attractors for such a system. Later, Temam and Ziane \cite{[TZ96]} used these operators to establish the global existence of strong solutions for Navier--Stokes equations in thin 3D domains with various boundary conditions when initial data belongs to ``large sets'' (see \cite{[TZ96]} for precise definition). They also used these operators to show that the average of the above obtained strong solution in the thin direction converges to the strong solution of a 2D Navier--Stokes system of equation as thickness $\eps$ goes to zero.

In this section we will recall these averaging operators according to our requirement and provide proofs of some of the results for the ease of the reader.

Let $M_\eps : C(Q_\eps; \R^3) \to C(Q; \R^2)$ be the averaging map; a map that averages the functions defined on thin domains $Q_\eps$ in the vertical direction and projects to functions defined on $Q$, and is defined by
\begin{equation}
\label{eqn-M_eps}
(M_\eps \psi)(\bx') := \frac{1}{\eps} \int_0^\eps \psi(\bx',x_3) dx_3, \qquad \bx' \in Q.
\end{equation}

\begin{lemma}
\lb{lemma2.2}
The map $M_\eps$ as defined in \eqref{eqn-M_eps} is continuous (and linear) with respect to norms $L^2(Q_\eps)$ and $L^2(\bbS)$. Moreover,
\begin{equation}
\lb{eq:2.2}
\|M_\eps \psi\|_{L^2(\bbS)}^2 \le \frac{1}{\eps}\|\psi\|^2_{L^2(Q_\eps)}, \qquad \psi \in L^2(Q_\eps).
\end{equation}
\end{lemma}

\begin{proof}
Take $\psi \in C({Q}_\eps)$ then by definition of $M_\eps$ we have
\[M_\eps \psi(\bx') = \dfrac{1}{\eps} \int_0^{\eps} \psi(\bx', x_3)\,dx_3.\]
Thus, using the Cauchy-Schwarz inequality we have
\begin{align*}
\|M_\eps \psi\|_{L^2(\bbS)}^2 = \int_\bbS |M_\eps \psi(\bx')|^2\,d\bx' &= \int_\bbS \left|\dfrac{1}{\eps} \int_{0}^{\eps}\psi(\bx',x_3)\,dx_3\right|^2\,d\bx' \\
& \le \dfrac{1}{\eps^2} \int_\bbS\left( \inteps |\psi(\bx', x_3)|^2\,dx_3 \inteps dx_3 \right)\,d \bx' \\
& = \dfrac{1}{\eps^2} \cdot \eps \inteps \int_\bbS |\psi(\bx', x_3)|^2\,dx_3\,d\bx' \\
& = \frac{1}{\eps} \|\psi\|^2_{L^2(Q_\eps)}.
\end{align*}
Therefore, we obtain
\begin{equation}
\lb{eq:2.3}
\|M_\eps \psi\|^2_{L^2(\bbS)} \le \frac{1}{\eps}\|\psi\|^2_{L^2(Q_\eps)},
\end{equation}
showing the map $M_\eps$ bounded and we can infer \eqref{eq:2.2}.
\end{proof}

\begin{corollary}
\lb{cor2.3}
The map $M_\eps$ as defined in \eqref{eqn-M_eps} has a unique extension,
which without the abuse of notation will be denoted by the same symbol $M_\eps \colon L^2(Q_\eps) \to L^2(\bbS)$.
\end{corollary}
\begin{proof}
Since $C(Q_\eps)$ is dense in $L^2(Q_\eps)$ and $M_\eps \colon L^2(Q_\eps) \to L^2(\bbS)$ is a bounded map thus by the Riesz representation theorem there exists a unique extension.
\end{proof}

\begin{remark}
\label{rem.M_eps_dual}
It is easy to check that the dual operator $M^*_\eps: L^2(Q) \rightarrow L^2(Q_\eps)$
is given by
\[
 (M^\ast_\eps \psi)(\bx',x_3) = \frac{1}{\eps} g(\bx'), \quad (\bx',x_3) \in Q_\eps.
\]
\end{remark}

Note that if $\psi$ is independent of $x_3$, i.e. $\psi(x',x_3) \equiv \xi(x')$ for all $x_3 \in (0,\eps)$
then $(M_\eps \psi)(\bx') = \xi(\bx')$ for $x' \in Q$.

We define another map $\what{M}_\eps$,
\begin{equation}\label{eqn-hatM_eps}
\begin{aligned}
& \what{M}_\eps : L^2(Q_\eps) \rightarrow L^2(Q_\eps) \\
& \what{M}_\eps \psi(\bx',y_3) = M_\eps \psi(\bx'), \qquad (\bx', y_3) \in Q_\eps,
\end{aligned}
\end{equation}
and for such a map we have the following scaling property:
\begin{lemma}\label{lem-2.2}
Let $\psi \in L^2(Q_\eps)$, then
\begin{equation}\label{eqn-scaling}
\|\what{M}_\eps \psi \|^2_{L^2(Q_\eps)} = \eps \|M_\eps \psi\|^2_{L^2(Q)}.
\end{equation}
\end{lemma}

\begin{proof}
Let $\psi \in L^2(Q_\eps)$. Then by the definition of the map $\what{M}_\eps$ and the Fubini Theorem, we have
\begin{align*}
\|\what{M}_\eps \psi\|^2_{L^2(Q_\eps)}
& = \int_{Q_\eps} |\what{M}_\eps \psi (\bx)|^2 d\bx = \int_{Q}\inteps |\what{M}_\eps \psi (\bx', y_3)|^2 d\bx' dy_3 \\
& = \int_Q \int_0^\eps |M_\eps \psi(\bx')|^2 d\bx' dy_3 = \eps \| M_\eps \psi\|^2_{L^2(Q)}.
\end{align*}
\end{proof}

The orthogonal component $\what{N}_\eps$, of the planar projection map $\what{M}_\eps$ is defined by
\begin{equation}
    \label{eqn-Nhat_eps}
    \what{N}_\eps = \mathrm{Id} - \what{M}_\eps,
\end{equation}
i.e.
\[
\what{N}_\eps \colon L^2(Q_\eps) \ni \psi \mapsto \psi - \what{M}_\eps \psi \in L^2(Q_\eps),
\]
where $\mathrm{Id}$ is the identity on $L^2(Q_\eps)$.

\subsection{Properties of the maps}
In the following lemma we establish an important property of the map $\what{N}_\eps$.
\begin{lemma}
\label{lemma_avgN}
Let $\psi \in L^2(Q_\eps)$, then
\begin{equation}
    \label{eqn-Nhat_eps-prop1}
\int_0^\eps \what{N}_\eps \psi(\bx', x_3)\,dx_3 = 0, \quad \mbox{ for a.a. } \bx^\prime \in  Q .
\end{equation}
\end{lemma}

\begin{proof}
Let us choose and fix $\psi \in L^2(Q_\eps)$.  Then by the definitions of the operators involved we have the following equality in $L^2(\bbS)$:
\[
\inteps \widehat{N}_\eps \psi(\bx', x_3)\,dx_3 = \eps \left[M_\eps \widehat{N}_\eps \psi\right](\bx').
\]
Therefore ,we deduce that in order to prove equality \eqref{eqn-Nhat_eps-prop1}, it is sufficient to show that
\[
M_\eps \widehat{N}_\eps = 0.
\]
Hence, by taking into account definitions \eqref{eqn-Nhat_eps} of $\widehat{N}_\eps$, we infer that it is sufficient to prove that
\[
M_\eps \circ \what{M}_\eps = M_\eps.
\]
Let us choose $\psi \in C(Q_\eps)$  and put  $\phi = \what{M}_\eps \psi$, i.e.
\[ \phi(\bx', y_3) = \frac{1}{\eps}\inteps \psi(\bx', x_3)\,dx_3, \qquad y_3 \in (0,\eps).\]
Thus,
\begin{align*}
[M_\eps \phi](\bx') & = \frac{1}{\eps}\inteps \phi(\bx', y_3)\,dy_3 = \frac{1}{\eps}\inteps \left(\frac{1}{\eps}\inteps \psi(\bx', x_3)\,dx_3\right)dy_3 \\
&   = \frac{1}{\eps}\inteps M_\eps \psi(\bx')\,dy_3 = M_\eps \psi (\bx'), \qquad  \bx' \in \bbS.
\end{align*}
Thus, we proved $M_\eps \circ\what{M}_\eps \psi = M_\eps \psi$ for every $\psi \in C({Q}_\eps)$.
Since $C({Q}_\eps)$ is dense in $L^2(Q_\eps)$ and the maps $M_\eps$ and $M_\eps \circ \what{M}_\eps$ are bounded in $L^2(Q_\eps)$, we conclude that we have  proved \eqref{eqn-Nhat_eps-prop1}.
\end{proof}

In the following two lemmas we enlist some properties of operators $\what{M}_\eps$ and $\what{N}_\eps$ which we will use repeatedly later in this and following sections.

\begin{lemma}
\label{lem.avg_op_sca_prop1}
Let $\eps > 0$ and $\what{M}_\eps$, $\what{N}_\eps$ be the maps defined as in \eqref{eqn-hatM_eps} and \eqref{eqn-Nhat_eps} respectively. Then for $\psi \in L^2(Q_\eps)$
\begin{align}
    \label{eqn-Mhat-squ}
    &\what{M}_\eps\left(\what{M}_\eps \psi \right) = \what{M}_\eps \psi, \\
    \label{eqn-Nhat-squ}
    &\what{N}_\eps\left(\what{N}_\eps \psi\right) = \what{N}_\eps \psi, \\
    \label{eqn-Mhat-Nhat-perp}
    &\what{M}_\eps\left(\what{N}_\eps \psi\right) = 0, \qquad \mbox{and} \qquad \what{N}_\eps \left(\what{M}_\eps \psi \right) = 0.
\end{align}
\end{lemma}

\begin{proof}
Let $\psi \in C(Q_\eps)$. Put $\phi = \what{M}_\eps \psi$, i.e for every $y_3 \in (0,\eps)$
\[
\phi(\bx', y_3) = [M_\eps \psi](\bx') = \frac{1}{\eps}\int_0^\eps \psi(\bx',x_3)\,dx_3.
\]
Next for $z_3 \in (0,\eps)$,
\begin{align*}
    [\what{M}_\eps \phi](\bx', z_3) & = [M_\eps \phi](\bx') = \frac{1}{\eps} \inteps \phi(\bx', y_3)\,dy_3 \\
    & = \frac{1}{\eps} \inteps \left( \frac{1}{\eps} \inteps \psi(\bx', x_3)\,dx_3\right) dy_3 \\
    & = \frac{1}{\eps}\inteps M_\eps\psi (\bx')\,dy_3 = [\what{M}_\eps \psi](\bx', z_3).
\end{align*}
Hence, we proved \eqref{eqn-Mhat-squ} for every $\psi \in C(Q_\eps)$. Since $C(Q_\eps)$ is dense in $L^2(Q_\eps)$ and $\what{M}_\eps$ is a bounded map on $L^2(Q_\eps)$, the identity holds true for every $\psi \in L^2(Q_\eps)$.

\noindent \textbf{Proof of first part of \eqref{eqn-Mhat-Nhat-perp}}. Let $\psi \in C(Q_\eps)$. Put $\phi = \what{N}_\eps \psi \in C(Q_\eps)$. By Lemma~\ref{lemma_avgN}
\[
\inteps \phi(\bx',x_3)\,dx_3 = 0.
\]
Therefore, we infer that
\[
[\what{M}_\eps \phi](\bx', y_3) = \frac{1}{\eps}\inteps \phi(\bx',x_3)\,dx_3 = 0,
\]
for $\bx' \in Q$ and every $y_3 \in (0,\eps)$. Thus, we have established \eqref{eqn-Mhat-Nhat-perp}$_1$ for all $\psi \in C(Q_\eps)$. Using the density argument, we can extend it to all $\psi \in L^2(Q_\eps)$.

For \eqref{eqn-Nhat-squ}, by definition of $\what{N}_\eps$ and \eqref{eqn-Mhat-Nhat-perp}$_1$, we obtain for $\psi \in L^2(Q_\eps)$
\[
\what{N}_\eps\left(\what{N}_\eps \psi\right) = \what{N}_\eps \psi - \what{M}_\eps \left(\what{N}_\eps \psi\right) = \what{N}_\eps \psi.
\]
Now for the second part of \eqref{eqn-Mhat-Nhat-perp} again by the definition of $\what{N}_\eps$ and \eqref{eqn-Mhat-squ}, for $\psi \in L^2(Q_\eps)$:
\[
\what{N}_\eps\left(\what{M}_\eps \psi\right) = \what{M}_\eps \psi - \what{M}_\eps\left(\what{M}_\eps \psi\right) = 0.
\]
\end{proof}

\begin{lemma}
\label{lem.na_Mhat}
Let $\eps > 0$ and $\psi \in H^1(Q_\eps)$. Then $\what{M}_\eps \psi$, $\what{N}_\eps \psi \in H^1(Q_\eps)$. Moreover
\[\|\nabla \what{M}_\eps \psi\|^2_{\lL^2(Q_\eps)} \le \|\nabla \psi\|^2_{\lL^2(Q_\eps)}.\]
\end{lemma}

\begin{proof}
Let $\eps > 0$ and $\psi \in H^1(Q_\eps)$. The proof is a direct consequence of the definition of $\what{M}_\eps$, $\what{N}_\eps$, the scaling property \eqref{eqn-scaling} and the estimate \eqref{eq:2.3}.
\end{proof}

\begin{lemma}
\label{lem.avg_op_sca_prop2}
Let $\eps>0$ and $\what{M}_\eps$, $\what{N}_\eps$ be the maps defined as in \eqref{eqn-hatM_eps} and \eqref{eqn-Nhat_eps} respectively. Then
\begin{enumerate}
\item[(i)] $\what{M}_\eps$ is self-adjoint on $L^2(Q_\eps)$, i.e.
\[
(\what{M}_\eps \psi, \xi)_{L^2(Q_\eps)} = (\psi, \what{M}_\eps \xi)_{L^2(Q_\eps)},
\qquad \psi, \xi \in L^2(Q_\eps).
\]
\item[(ii)] The ranges of $\what{M}_\eps$ and $\what{N}_\eps$ are perpendicular, i.e.
\[
 (\what{M}_\eps \psi, \what{N}_\eps \xi)_{L^2(Q_\eps)} = 0,
\qquad \psi, \xi \in L^2(Q_\eps).
\]
\item[(iii)] If $\psi, \xi \in H^1(Q_\eps)$ then
\[
 \left(\nabla\what{M}_\eps \psi, \nabla\what{N}_\eps \xi\right)_{\lL^2(Q_\eps)} = 0.
\]
\end{enumerate}
\end{lemma}

\begin{proof}
\noindent\textbf{(i)} Let $\psi , \xi \in L^2(Q_\eps)$, then by the definition of $\what{M}_\eps$:
\begin{align*}
(\what{M}_\eps \psi, \xi)_{L^2(Q_\eps)}
&= \int_{Q_\eps} [\what{M}_\eps \psi] (\bx) \xi(\bx)\,d\bx \\
&= \int_Q \inteps [M_\eps \psi](\bx') \xi(\bx',x_3)\,d\bx'dx_3 \\
&= \int_Q \left(\frac{1}{\eps}\inteps \psi (\bx',x_3)\,dx_3\right) \left(\inteps \xi(\bx',x_3)\,dx_3\right)\,d\bx' \\
&= \int_Q \inteps \psi (\bx',x_3) \left(\frac{1}{\eps}\inteps \xi(\bx',x_3)\,dx_3\right)\,d\bx' dx_3\\
& = \int_Q\inteps \psi(\bx',x_3) M_\eps \xi(x')\,d\bx'dx_3 = (\psi, \what{M}_\eps \xi)_{L^2(Q_\eps)},
\end{align*}
showing that $\what{M}_\eps$ is self adjoint on $L^2(Q_\eps)$.

\noindent \textbf{(ii)} Recall from \eqref{eqn-Mhat-squ} that $\what{M}_\eps \circ \what{M}_\eps = \what{M}_\eps$ and from (i) $\what{M}_\eps$ is self-adjoint. Therefore, for $\psi ,\xi \in L^2(Q_\eps)$, using the definition of $\what{N}_\eps$:
\begin{align*}
(\what{M}_\eps \psi, \what{N}_\eps \xi)_{L^2(Q_\eps)} & =
 (\what{M}_\eps \psi, \xi)_{L^2(Q_\eps)} - (\what{M}_\eps \psi, \what{M}_\eps \xi)_{L^2(Q_\eps)} \\
& = \left(\what{M}_\eps \psi, \xi\right)_{L^2(Q_\eps)} - \left(\what{M}_\eps\left(\what{M}_\eps \psi \right), \xi\right)_{L^2(Q_\eps)} = 0,
\end{align*}
establishing the orthogonality of the ranges of $\what{M}_\eps$ and $\what{N}_\eps$.

\noindent \textbf{(iii)}
Let $\psi, \xi \in H^1(Q_\eps)$. First note that, since $\what{M}_\eps \psi$ is independent of $x_3$, we have
\[
\nabla \what{M}_\eps \psi
 = \left(\partial_1 \what{M}_\eps \psi, \partial_2 \what{M}_\eps \psi, \partial_3 \what{M}_\eps \psi \right)
 = \left(\what{M}_\eps \partial_1 \psi,  \what{M}_\eps \partial_2 \psi, 0 \right).
\]
Similarly,
\[
\nabla \what{N}_\eps \xi = \left( \what{N}_\eps \partial_1\xi, \what{N}_\eps\partial_2 \xi, \partial_3 g\right).
\]
Hence, from (ii) we get
\[
\left(\nabla \what{M_\eps} \psi, \nabla \what{N}_\eps \xi\right)_{\lL^2(Q_\eps)}
 = \left(\what{M}_\eps \partial_1\psi, \what{N}_\eps\partial_1\xi \right)_{L^2(Q_\eps)}
    + \left( \what{M}_\eps \partial_2 \psi, \what{N}_\eps \partial_2 \xi\right)_{L^2(Q_\eps)} + 0 = 0.
\]
\end{proof}

A direct consequence of Lemma~\ref{lem.na_Mhat} and Lemma~\ref{lem.avg_op_sca_prop2} is the following corollary.

\begin{corollary}
\label{cor.Mhat-Nhat-pyth}
Let $\eps > 0$ and $\psi \in L^2(Q_\eps)$, $\xi \in H^1(Q_\eps)$. Then
\begin{equation}
    \label{eqn-Mhat-Nhat-pyth}
    \begin{split}
        \|\psi\|^2_{L^2(Q_\eps)} &= \|\what{M}_\eps \psi\|^2_{L^2(Q_\eps)} + \|\what{N}_\eps \psi\|^2_{L^2(Q_\eps)},\\
        \|\nabla \xi\|^2_{L^2(Q_\eps)} &= \|\nabla \what{M}_\eps \xi\|^2_{L^2(Q_\eps)} + \|\nabla \what{N}_\eps \xi\|^2_{L^2(Q_\eps)}.
    \end{split}
\end{equation}
\end{corollary}

Next we introduce the averaging operators for $\R^3$-valued vector fields using the maps defined above (for scalar functions) as follows$\colon$
\begin{equation}\label{eqn-M_eps-tilde}
\begin{split}
\wtd{M}_\eps\colon &\lL^2(Q_\eps) \rightarrow \lL^2(Q_\eps) \\
 &u \mapsto (\what{M}_\eps u_1, \what{M}_\eps u_2, 0)
 \end{split}
 \end{equation}
and
\begin{equation}
    \label{eqn-N_eps-tilde}
    \wtd{N}_\eps = \mathrm{Id} - \wtd{M}_\eps \in \mathcal{L}\left(\lL^2(Q_\eps)\right),
\end{equation}
where $\mathrm{Id}$ is the identity map on $\LQeps$.

If $u$ is independent of $x_3$ and lives on $\R^2$, i.e. $u(\bx',x_3) = (u_1(\bx'), u_2(\bx'), 0)$ then
\[\wtd{M}_\eps u(\bx',x_3) = (u_1(\bx'), u_2(\bx'), 0) = u(\bx',x_3) \quad \mbox{ for } \bx' \in Q.\]

\begin{lemma}\label{lemm2_1}
If a vector field $u:Q_\eps \rightarrow \R^3$ satisfies the boundary conditions
\eqref{eqn-snse3}, i.e.
\begin{align*}
&u_3 = 0, \; \partial_3 u_j = 0 \text{ for } j = 1, 2 \text{ on } \Gamma^\eps_h \mbox{ and }\, u = 0 \text{ on } \Gamma^\eps_\ell
\end{align*}
then
\begin{align}
& M_\eps u_1 = M_\eps u_2 = 0 \text{  on } \partial Q,\label{eqn-int3Meps} \\
& \wtd{N}_\eps u \cdot \vec{n}_\eps = 0 \text{ on } \Gamma_{h}^\eps.
\label{eqn-int3Neps}
\end{align}
\end{lemma}
\begin{proof}
Since $u=0$ on $\Gamma_\ell^{\eps} = \partial Q \times [0,\eps]$, so
$u_1 = u_2 = 0$ on $\partial Q \times [0,\eps]$ and hence
$M_\eps u_1 = M_\eps u_2 = 0$ on $\partial Q$. Now for \eqref{eqn-int3Neps}, using the definition of $\wtd{N}_\eps$, $\wtd{M}_\eps$ we get
\[\wtd{N}_\eps u = \left(\what{N}_\eps u_1, \what{N}_\eps u_2, u_3\right),\]
and observing that $\vec{n}_\eps = (0,0,\pm 1)$ on $\Gamma_h^{\eps,\pm}$, $u_3 = 0$ on $\Gamma_h^\eps$ we conclude that $\wtd{N}_\eps u \cdot \vec{n}_\eps = 0$.
\end{proof}

\begin{lemma}
\label{lem.div-free}
Let $u \in \rH_\eps$, then
\[
\ddiv \wtd{M}_\eps u = 0, \quad \mbox{ and } \quad \ddiv \wtd{N}_\eps u = 0 \qquad \mbox{in } Q_\eps.
\]
\end{lemma}

\begin{proof}
Let
\[\rU_\eps := \left\{ u\in C^\infty(Q_\eps): \ddiv u = 0 \mbox{ in } Q_\eps \mbox{ and } u \cdot \vec{n}_\eps = 0 \mbox{ on } \partial Q_\eps\right\}, \]
and choose $u \in \rU_\eps$. Then by the definition of $\wtd{M}_\eps$, we have for $\bx = (\bx',x_3) \in Q_\eps$
\begin{align*}
    \ddiv \wtd{M}_\eps u (\bx) &= \partial_1 \what{M}_\eps u_1 + \partial_2 \what{M}_\eps u_2 + \partial_3 0 \\
    & = \what{M}_\eps \left[\partial_1 u_1 + \partial_2 u_2\right] \\
    & = - \what{M}_\eps \partial_3 u_3 = - \frac{1}{\eps} \inteps \partial_3 u_3(\bx',x_3)\,dx_3 \\
    & = -\frac{1}{\eps} \left[u_3(\bx',\eps) - u_3(\bx',0)\right] = 0,
\end{align*}
where we used that $\ddiv u = 0$ in $Q_\eps$ in the second last line and $u\cdot\vec{n}_\eps = 0$ on $\partial Q_\eps$ in the last line. Thus, we have proved that $\ddiv \wtd{M}_\eps = 0$ in $Q_\eps$ for $u \in \rU_\eps$. Since $\rU_\eps$ is dense in $\rH_\eps$, it holds true for every $u \in \rH_\eps$ too. The second part of the lemma for $\wtd{N}_\eps$ follows directly from the definition of $\wtd{N}_\eps$ and first part of the lemma.
\end{proof}

From Lemma~\ref{lemm2_1} and Lemma~\ref{lem.div-free}, we infer the following corollary:

\begin{corollary}
\label{cor.Heps}
If $u \in \rH_\eps$, then $\wtd{M}_\eps u$ and $\wtd{N}_\eps u$ belongs to $\rH_\eps$.
\end{corollary}

In the following lemma we extend the results from Lemma~\ref{lem.avg_op_sca_prop1}--Lemma~\ref{lem.avg_op_sca_prop2} to the averaging maps $\wtd{M}_\eps$ and $\wtd{N}_\eps$ defined for vector valued functions. The proofs of each of the statements in the lemma can be carried out analogously to the scalar case.

\begin{lemma}
\label{lem.avg_op_vec_prop}
Let $\eps > 0$ and $\wtd{M}_\eps$, $\wtd{N}_\eps$ be the maps as defined in \eqref{eqn-M_eps-tilde} and \eqref{eqn-N_eps-tilde} respectively. Then
\begin{enumerate}
    \item[(i)] for $u \in \lL^2(Q_\eps)$
\begin{align}
    \label{eqn-Mtilde-squ}
    &\wtd{M}_\eps\left(\wtd{M}_\eps u \right) = \wtd{M}_\eps u, \\
    \label{eqn-Ntilde-squ}
    &\wtd{N}_\eps\left(\wtd{N}_\eps u\right) = \wtd{N}_\eps u, \\
    \label{eqn-Mtilde-Ntilde-perp}
    &\wtd{M}_\eps\left(\wtd{N}_\eps u\right) = 0, \qquad \mbox{and} \qquad \wtd{N}_\eps \left(\wtd{M}_\eps u \right) = 0.
\end{align}
    \item[(ii)] for $u \in \hH^1(Q_\eps)$, $\wtd{M}_\eps u$, $\wtd{N}_\eps u \in \hH^1(Q_\eps)$ and
\[
\|\nabla \wtd{M}_\eps u\|^2_{\lL^2(Q_\eps)} \le \|\nabla u\|^2_{\lL^2(Q_\eps)}.
\]
    \item[(iii)] $\wtd{M}_\eps$ is self-adjoint on $\lL^2(Q_\eps)$, i.e.
\[
(\wtd{M}_\eps u, \rv)_{\lL^2(Q_\eps)} = (u, \wtd{M}_\eps \rv)_{\lL^2(Q_\eps)},
\qquad u, \rv \in \lL^2(Q_\eps).
\]
\item[(iv)] The ranges of $\wtd{M}_\eps$ and $\wtd{N}_\eps$ are perpendicular, i.e.
\[
 (\wtd{M}_\eps u, \wtd{N}_\eps \rv)_{\lL^2(Q_\eps)} = 0,
\qquad u, \rv \in \lL^2(Q_\eps).
\]
\item[(v)] If $u, \rv \in \hH^1(Q_\eps)$ then
\[
 \left(\nabla\wtd{M}_\eps u, \nabla\wtd{N}_\eps \rv\right)_{\lL^2(Q_\eps)} = 0.
\]
\end{enumerate}
\end{lemma}

From Lemma~\ref{lem.avg_op_vec_prop}, we can deduce the following corollary.
\begin{corollary}\label{cor-Pythagoras}
If $u \in \lL^2(Q_\eps)$, then
\begin{equation}\label{eqn-pythagoras1}
 \|\wtd{M}_\eps u \|^2_{\lL^2(Q_\eps)} +
  \|\wtd{N}_\eps u\|^2_{\lL^2(Q_\eps)} = \|u\|^2_{\lL^2(Q_\eps)}.
\end{equation}
If $u \in \hH^{1}(Q_\eps)$, then
\begin{equation}\label{eqn-pythagoras2}
\|\nabla \wtd{M}_\eps u \|^2_{\lL^2(Q_\eps)} +
 \|\nabla \wtd{N}_\eps u\|^2_{\lL^2(Q_\eps)} = \|\nabla u\|^2_{\lL^2(Q_\eps)}.
\end{equation}
\end{corollary}

Finally, we define a map that projects $\R^3$-valued vector fields defined on $Q_\eps$ to $\R^2$-valued vector fields defined on $Q$.
\begin{equation}\label{eqn-M_eps-0}
\begin{split}
 \ocirc{M}_\eps\colon  &\lL^2(Q_\eps) \rightarrow \lL^2(Q),\\
   & u \mapsto (M_\eps u_1, M_\eps u_2).
   \end{split}
\end{equation}
We remark that while the range of $\wtd{M}_\eps$ equals to $\lL^2(Q_\eps, \R^2)$,
the range of $\ocirc{M}_\eps$ equals to $\lL^2(Q, \R^2)$.

Analogous to the scaling property for scalar functions and corresponding averaging maps we have a scaling property for $\tMe$ and $\oMe$.

\begin{lemma}
\label{lem.vec_scaling}
Let $\eps > 0$ and  $u \in \LQeps$. Then
\begin{equation}\label{eqn-vscaling}
 \| \wtd{M_\eps} u\|^2_{\lL^2(Q_\eps)} = \eps \|\ocirc{M}_\eps u\|^2_{\lL^2(Q)}, \qquad u \in \LQeps.
\end{equation}
\end{lemma}

\begin{proof}
Let $u \in \LQeps$, then by the definition of $\tMe$
\begin{align*}
    \|\tMe u\|^2_\LQeps & = \int_{Q_\eps}|\tMe u(\bx)|^2 d\bx  = \int_{Q_\eps}\left( |\what{M}_\eps u_1 (\bx)|^2 + |\what{M}_\eps u_2 (\bx)|^2\right) d\bx \\
    & = \inteps \int_Q \left(|M_\eps u_1(\bx')|^2 + |M_\eps u_2(\bx')|^2\right)d\bx' dx_3 \\
    & = \eps \int_Q |\oMe u (\bx')|^2d\bx' = \eps\|\oMe u\|^2_\LS.
\end{align*}
\end{proof}

\begin{remark}
\label{rem.circM-dual}
Similar to the scalar case, one can prove that the dual
operator $\ocirc{M}^*_\eps: \lL^2(Q) \rightarrow \lL^2(Q_\eps)$ is given by
\begin{equation}\label{circM-dual}
\ocirc{M}^*_{\eps}\rv (\bx',x_3) =  \big(\frac{1}{\eps} \rv_1(\bx'), \frac{1}{\eps} \rv_2(\bx'),0\big), \qquad (\bx',x_3) \in Q_\eps.
\end{equation}
Indeed, for $u \in \LQeps$ and $\rv \in \lL^2(Q)$,
\begin{align*}
(\ocirc{M}_\eps u, \rv)_{\lL^2(Q)} & = (M_\eps u_1, \rv_1)_{L^2(Q)} + (M_\eps u_2, \rv_2)_{L^2(Q)} \\
   & = (u_1, M^*_\eps \rv_1)_{L^2(Q_\eps)} + (u_2, M^*_\eps \rv_2)_{L^2(Q_\eps)} + (u_3,0).
\end{align*}
\end{remark}


\begin{lemma}\label{lemma:M N}
Let $\eps>0$ and $\oMe$ is the map as defined in \eqref{eqn-M_eps-0}. Then
\begin{enumerate}
\item[(i)]
if $u \in \rH_\eps$ then $\ocirc{M}_\eps u \in \rH$, i.e. $\ocirc{M}_\eps : \rH_\eps \to \rH$.
\item[(ii)]
if $u \in \rV_\eps$ then $\ocirc{M}_\eps u \in \rV$, i.e. $\ocirc{M}_\eps : \rV_\eps \rightarrow \rV$.
\end{enumerate}
\end{lemma}

\begin{proof}
\noindent \textbf{(i)} Let us take $u \in \rH_\eps$. Then $u\cdot \vec{n}_\eps = 0$ on $\partial Q_\eps$, where $\vec{n}_\eps(\bx',x_3) = (\vec{n}(\bx'),0)$ for $(\bx',y) \in \Gamma^\eps_\ell$ and
$\vec{n}_\eps = (0,0,\pm 1)$ on $\Gamma^{\eps,\pm}_h$. Put $\rv = \ocirc{M}_\eps u = (M_\eps u_1, M_\eps u_2)$. Thus, for almost all $x' \in \partial Q$,
\begin{align*}
\rv \cdot \vec{n}(\bx') & = M_\eps u_1(\bx'){n}_1(\bx') + M_\eps u_2(\bx') {n}_2(\bx') \\
& = \left(\frac{1}{\eps} \int_0^\eps u_1(\bx',x_3)\,dx_3\right) {n}_1(\bx') + \left(\frac{1}{\eps} \inteps u_2(\bx',x_3)\,dx_3\right) {n}_2(\bx')\\
& = \frac{1}{\eps} \int_0^\eps \left(u_1(\bx',x_3) {n}_1(\bx') + u_2(\bx',x_3) {n}_2(\bx') \right) dx_3\\
& = \frac{1}{\eps} \int_0^\eps u(x',x_3) \cdot \vec{n}_\eps(x',x_3) dx_3 = 0,
\end{align*}
where $\vec{n}$ is the unit external normal vector field to $\partial Q$.
The divergence free condition $\ddiv' \rv=0$ can be shown analogously to the proof of Lemma~\ref{lem.div-free}. Thus, we have shown that $\oMe u \in \rH$.

\noindent \textbf{(ii)} Let us begin with $\ocirc{M}_\eps : \lL^2(Q_\eps) \rightarrow \lL^2(Q)$ and put
\[
\rv = \ocirc{M}_\eps u = \left(M_\eps u_1, M_\eps u_2 \right).
\]
It is obvious that if $\partial_j u_k \in L^2(Q_\eps)$ for $k=1,2,3$, $j=1,2$ then $M_\eps\left(\partial_j u_k\right) \in L^2(Q)$
and $\partial_j \rv = \ocirc{M}_\eps (\partial_j u)$ a.e. Now let  $u \in \rV_\eps$. Then from the above
reason $\rv =  \ocirc{M}_\eps u \in \hH^1(Q)$. Moreover, the proof of part (i) implies that
$\ddiv' \rv = 0$. Finally, we have $u_j(\bx',x_3) = 0$ for $(\bx',x_3) \in \partial Q \times (0,\eps)$, $j = 1, 2$, therefore for $\bx' \in \partial Q$
\[
  \rv_j(\bx') = M_\eps u_j(\bx') = \frac{1}{\eps} \int_0^\eps u_j (\bx',x_3)\, dx_3 = 0\qquad  j=1, 2.
\]
Hence, we have proved that $\rv \in \rV$.
\end{proof}

\begin{remark}\label{rem:dualMo}
We proved that $\ocirc{M}_\eps (\rH_\eps) \subset \rH$.
By \eqref{circM-dual}, for $\rv \in \rH$, and $u = \ocirc{M}_\eps^* \rv$, $\ddiv u = 0$.
Moreover, $u_3 = 0$ on $\Gamma^\eps_h$ and
\[
 u \cdot \vec{n}_\eps(\bx',x_3) = \frac{1}{\eps} \rv \cdot \vec{n}(\bx') =  0, \quad \mbox{for }\, (\bx', x_3) \in \Gamma^\eps_\ell.
\]
Thus, we infer $u \in \rH_\eps$ which implies $\ocirc{M}^*_\eps(\rH) \subset \rH_\eps$.
\end{remark}

Let us now recall a classical Poincar\'{e}'s inequality, see e.g.
\cite[Proposition 2.1]{[TZ96]}
\begin{proposition}[Poincar\'e inequality]
\label{prop-Poincare}
If $\psi \in H^1(Q_\eps)$ and either of the following conditions are satisfied
\[
\begin{cases}
\psi = 0\; &\mbox{ on }\, \Gamma^{\eps,+}_h, \\
\psi = 0\; &\mbox{ on }\,  \Gamma^{\eps,-}_h,\\
 M_\eps \psi = 0\; &\mbox{ a.e. on }\, Q.
\end{cases}
\]
Then
\begin{equation}\label{eqn-Poincare}
 \|\psi\|_{L^2(Q_\eps)} \le \eps \|\partial_3 \psi\|_{L^2(Q_\eps)}.
\end{equation}
\end{proposition}

We have the following immediate consequence.
\begin{corollary}\label{cor-Poincare}
If $u \in \rV_\eps$, then
\begin{equation}\label{NuV}
\|\wtd{N}_\eps u\|_{\lL^2(Q_\eps)} \le \eps \|\partial_3 \wtd{N}_\eps u\|_{\lL^2(Q_\eps)}.
\end{equation}
\end{corollary}

\begin{proof}
Note that for $u \in \rV_\eps$, $\wtd{N}_\eps u = (\what{N}_\eps u_1, \what{N}_\eps u_2, u_3)$. Thus, it is sufficient to prove an analogous inequality for $\what{N}_\eps$ and $\psi \in H^1(Q_\eps)$ as we can use Proposition~\ref{prop-Poincare} directly for $u_3$ (since for $u \in \rV_\eps$, $u_3 = 0$ on $\Gamma_h^\eps$). For such $\psi$ we have, by \eqref{eqn-Mhat-Nhat-perp}, $\what{M}_\eps \what{N}_\eps \psi = 0$ a.e. on $Q$ and by Lemma~\ref{lem.na_Mhat} $\what{N}_\eps \psi \in H^1(Q_\eps)$. Hence, by Proposition~\ref{prop-Poincare} the inequality holds true for the map $\what{N}_\eps$ and $\psi \in H^1(Q_\eps)$. Therefore, we have shown that \eqref{eqn-Poincare} holds true for every component of $\tNe u$ and in particular for $\tNe u$.
\end{proof}

We also have the following anisotropic Ladyzhenskaya's inequality (see \cite[Remark~2.1]{[TZ96]}).
\begin{proposition}\label{prop-Lady-anisotropic}
Let $\eps > 0$ and $u \in \hH^1(Q_\eps)$. Then, $\exists$ $c_0 = c_0(Q) >0$ independent of $\eps$ such that
\begin{equation}\label{eqn-lady-anisotropic}
\|u\|_{\lL^6(Q_\eps)} \le c_0
\left( \frac{1}{\eps} \|u\|_{\lL^2(Q_\eps)}
   +\|\partial_3 u\|_{\lL^2(Q_\eps)} \right)^{1/3}
   (\|u\|_{\lL^2(Q_\eps)}+ \|\partial_1 u\|_{\lL^2(Q_\eps)} +
   \|\partial_2 u\|_{\lL^2(Q_\eps)} )^{2/3}.
\end{equation}
\end{proposition}

From Corollary~\ref{cor-Poincare} and Proposition~\ref{prop-Lady-anisotropic}
we deduce the following result.
\begin{corollary}\label{cor-Lady-anisotropic}
Let $\eps >0$ and $u \in \rV_\eps$. Then, $\exists$ $c_0 > 0$ independent of $\eps$ such that
\[
 \|\wtd{N}_\eps u\|_{\lL^6(Q_\eps)} \le c_0 \| \wtd{N}_\eps u\|_{\rV_\eps},
\]
where
\[
 \|u\|^2_{\rV_\eps} = \|u\|^2_{\lL^2(Q_\eps)} + \| \nabla u \|^2_{\lL^2(Q_\eps)}.
\]
\end{corollary}

\begin{proof}
Let $\eps > 0$ and $u \in \rV_\eps$. Then by Proposition~\ref{prop-Lady-anisotropic} there exists a constant $c_0$ independent of $\eps$ such that
\begin{align}
\label{eqn-cor-Lady1}
    \|\tNe u\|_{\lL^6(Q_\eps)} & \le c_0
\left( \frac{1}{\eps} \|\tNe u\|_{\lL^2(Q_\eps)}
   +\|\partial_3 \tNe u\|_{\lL^2(Q_\eps)} \right)^{1/3}
   \left(\|\tNe u\|_{\lL^2(Q_\eps)}+ \|\nabla'\tNe u\|_{\lL^2(Q_\eps)} \right)^{2/3}.
\end{align}
Now from Corollary~\ref{cor-Poincare} for $u \in \rV_\eps$
\[\frac{1}{\eps} \|\tNe u\|_{\LQeps} \le \|\partial_3 \tNe u\|_\LQeps.\]
Using the above inequality in \eqref{eqn-cor-Lady1} and the definition of $\rV_\eps$-norm we obtain the required inequality.
\end{proof}

From Corollaries~\ref{cor-Poincare}, \ref{cor-Lady-anisotropic} and interpolation argument (between $\lL^2(Q_\eps)$ and $\lL^6(Q_\eps)$)\footnote{Let $u \in \lL^6(Q_\eps)$ then there exists a constant $c > 0$ such that for $2 \le q \le 6$, $\|u\|_{\lL^q}^2 \le c \|u\|^{\frac{6-q}{q}}_{\lL^2}\|u\|^{\frac{3(q-2)}{q}}_{\lL^6}$.},
we obtain
\begin{equation}\label{eqn-NepsL3}
\|\wtd{N}_\eps u\|^2_{L^3(Q_\eps)} \le c_0 \eps \| \wtd{N}_\eps u\|^2_{\rV_\eps}.
\end{equation}

We have defined a map $\ocirc{M}_\eps$ (see \eqref{eqn-M_eps-0}) from the set of (square integrable) $\R^3$-valued vector fields on $Q_\eps$ to the set of (square integrable) $\R^2$-valued vector fields on $Q$. Later in the analysis we will require a retract of $\ocirc{M}_\eps$, i.e. a map $R_\eps : \lL^2(Q) \rightarrow \lL^2(Q_\eps)$
such that
\begin{equation} \label{eqn-4_17}
    \ocirc{M}_\eps \circ R_\eps = \text{Id on }  \lL^2(Q).
\end{equation}

A natural candidate for $R_\eps$ is
\begin{equation}\label{eqn-4_18}
R_\eps u'(\bx', x_3) = \left( u_1'(\bx'), u_2'(\bx'), 0\right),
\end{equation}
where $\lL^2(Q, \R^2) \ni u' = \left(u_1', u_2'\right)$. Note that $R_\eps$ is a bounded linear map from $\lL^2(Q)$ to
$\lL^2(Q_\eps)$ and $\ocirc{M}_\eps \circ R_\eps =$ Id on $ \lL^2(Q)$.
Moreover, if $\ddiv' u' = 0$ then
\[
(\ddiv R_\eps u')(\bx',x_3) = \partial_1 u_1'(\bx') + \partial_2 u_2'(\bx') + \partial_3 0
 = \ddiv' u' = 0.
\]
Moreover, if $u' \cdot \vec{n}|_{\partial Q} = 0$ then it can be verified that $R_\eps u' \cdot \vec{n}_\eps = 0$ on $\partial Q_\eps$. Hence, for $u' \in \rH$, $R_\eps u' \in \rH_\eps$. Furthermore, if $u' \in \mathbb{H}^1(Q)$ then $R_\eps u' \in \mathbb{H}^1(Q_\eps)$ and similarly to the above argument we can show that  $u' \in \rV$ implies $R_\eps u' \in \rV_\eps$.

Indeed, if $u' = 0$ on $\partial Q$ then $R_\eps u' = 0$ on $\Gamma^\eps_{\ell}$ and
$(R_\eps u')_3 = 0$ on $\Gamma_h^\eps$. Finally, if $u' \in D(\rA)$ then
$R_\eps u' \in D(A_\eps)$. Indeed, if $u \in \mathbb{H}^2(Q)$ then $R_\eps u \in \mathbb{H}^2(Q_\eps)$
and for $j=1,2$, $(R_\eps u')_j$ is constant
with respect to $x_3$ and so $\partial_3( (R_\eps u')_j ) = 0$ also on $\Gamma^\eps_h$.

Also, note that $R_\eps$ commutes with $\nabla'$, i.e. for $u' \in \rV$, $\rv \in \rV_\eps$,
\begin{equation}\label{eqn:Rdelcommute}
\begin{aligned}
  \partial_j R_\eps u' &= R_\eps \partial_j u', \quad j=1,2,  \\
  \ocirc{M}_\eps \partial_j \rv &=  \partial_j \ocirc{M}_\eps \rv, \quad j=1,2.
\end{aligned}
\end{equation}
Additionally, for $\rv \in \lL^2(Q_\eps)$ and $u' \in \lL^2(Q)$
\begin{equation}\label{eqn-4_19}
\left(R_\eps u', \rv\right)_{\lL^2(Q_\eps)} = \eps \left(u', \ocirc{M}_\eps \rv\right)_{\lL^2(Q)},
\end{equation}
showing that the adjoint $R_\eps^\ast : \lL^2(Q_\eps) \to \lL^2(Q)$ of $R_\eps$ is given by $R_\eps^\ast \rv = \eps \ocirc{M}_\eps \rv$.

\begin{lemma}
\label{lemma2.22}
Let $\eps > 0$ and $\rv \in \lL^2(Q)$. Then
\begin{equation}
\label{eq:2.37}
\|R_\eps \rv\|_{\LQeps}^2 = \eps \|\rv\|_{\LS}^2\,, \quad \rv \in \LS\;.
\end{equation}
\end{lemma}

\begin{proof}
Let $\eps > 0$ and consider $\LS \ni \rv = \left(\rv_1,\rv_2\right)$. Then, by the definition of the retract operator $R_\eps$ and $\LQeps$-norm we have
\begin{align*}
\|R_\eps \rv\|^2_{\LQeps}& = \int_{Q}\int_0^\eps\left|R_\eps \rv(\bx',x_3)\right|^2 d\bx'dx_3 = \int_{Q}\inteps \left[|\rv_1(x')|^2+|\rv_2(x')|^2 \right]d\bx'dx_3 \\
& = \eps \|\rv\|^2_{\LS}\;.
\end{align*}
\end{proof}

Using the definition of the map $R_\eps$ and its properties the above lemma can be generalised to Sobolev spaces.

\begin{lemma}
\label{lem:Reps_Sob}
Let $\eps >0$, $p \ge 2$ and $R_\eps$ be the map as given by \eqref{eqn-4_17}--\eqref{eqn-4_18}. Then
\begin{equation}
\label{eq:Reps_Sob}
    \begin{split}
        \|\nabla R_\eps \rv\|_{\lL^p(Q_\eps)} = \eps^{1/p}\|\nablaS \rv\|_{\lL^p(Q)}, \qquad \rv \in \mathbb{W}^{1,p}(Q), \\
        \|\Delta R_\eps \rv\|_{\LQeps} = \eps^{1/2}\|\Delta' \rv\|_\LS, \qquad \rv \in \mathbb{W}^{2,2}(Q).
    \end{split}
\end{equation}
\end{lemma}



\section{The stochastic NSEs on a thin domain}
\label{sec:snse}

This section deals with the proof of our main result, Theorem~\ref{thm:main_thm}. First we introduce our two systems; stochastic NSE in thin domain $Q_\eps$ and 2D stochastic NSE on $Q$, then we present the definition of martingale solutions for both systems. We also state the assumptions under which we prove our result. In \S\,\ref{sec:estimates_thin_film}, we obtain a priori estimates (formally) which we further use to establish some tightness criterion (see \S\,\ref{sec:tightness_thin_film}) which along with Jakubowski's generalisation of Skorokhod Theorem gives us a converging (in $\varepsilon$) subsequence. At the end of this section we show that the limiting object of the previously obtained converging subsequence is a martingale solution of 2D stochastic NSE on $Q$ (see \S\,\ref{sec:proof-theorem}).

In thin domains $Q_\eps$, which was introduced in \eqref{eqn-Q-eps}, we consider the following stochastic Navier--Stokes equations (SNSE)
\begin{align}
 d \tue - [ \nu \Delta \tue - (\tue \cdot \nabla) \tue - \nabla \widetilde{p}_\eps]dt = \tfe dt +  \widetilde{G}_\eps\, d\tWe(t) &\quad \text{ in } Q_\eps \times (0,T), \label{eq:5.1}\\
\ddiv \tue = 0 &\quad \text{ in } Q_\eps \times (0,T),\label{eq:5.2} \\
\wtd{u}_{\eps,3} = 0, \quad \partial_3 \wtd{u}_{\eps,j}=0, \;\; j=1,2
&\quad\text{ on } \Gamma^\eps_h \times (0,T), \label{eq:5.3} \\
\wtd{u}_\eps = 0 &\quad \text{ on  } \Gamma^\eps_{\ell} \times (0,T),\label{eq:5.3a} \\
\wtd{u}_\eps(\cdot,0) = \wtd{u}^\eps_0 &\quad\text{ in } Q_\eps, \label{eq:5.4}
\end{align}
where $\Gamma_h^\eps$ and $\Gamma_\ell^\eps$ were introduced in \eqref{eqn-Gamma}.
Recall that, $\tue=(\tu_{\eps,1}, \tu_{\eps,2}, \tu_{\eps,3})$ is the fluid velocity field, $p$ is the pressure, $\nu>0$ is a (fixed) kinematic viscosity, $\tu_0^\eps$ is a divergence free vector field on $Q_\eps$ and $\vec{n}$ is the unit normal outer vector to the boundary $\partial Q_\eps$.
We assume that\footnote{We could have considered the case $N = \infty$, but it doesn't give any additional mathematical novelty on the other hand makes notations cumbersome.} $N \in \N$.
We consider a family of maps
\[\widetilde{G}_\eps : \R_+  \to \mathcal{T}_2(\R^N; \rH_\eps)\]
such that
\begin{equation}
\label{eq:5.5}
\widetilde{G}_\eps(t) k  := \sum_{j=1}^N \widetilde{g}^j_\eps(t) k_j,\;\; k=(k_j)_{j=1}^N \in \mathbb{R}^N.
\end{equation}
for some  $\widetilde{g}_\eps^j:  \R_+  \to \rH_\eps$, $j=1, \cdots,N$. The Hilbert-Schmidt norm of $\widetilde{G}_\eps$ is given by
\begin{equation}
\label{eq:HSnorm}
\|\widetilde{G}_\eps(s)\|^2_{\mathcal{T}_2(\R^N; \rH_\eps)} = \sum_{j=1}^N \|\widetilde{g}^j_\eps(s)\|^2_{\LQeps}.
\end{equation}

Finally  we assume that  $\tWe(t)$, $t \ge 0$ is an $\R^N$-valued Wiener process defined  on the probability space $\left(\Omega, \mathcal{F}, \mathbb{F}, \mathbb{P}\right)$. We assume that
 $\left(\beta_j\right)_{j=1}^N$ are i.i.d real valued Brownian motions such that  $W(t)=\bigl(\beta_j(t)\bigr)_{j=1}^N$, $t\geq 0$.

In this section we shall establish convergence of the radial averages of the martingale solution of the 3D stochastic equations \eqref{eq:5.1}--\eqref{eq:5.4}, as the thickness of the domain $\eps \rightarrow 0$, to a martingale solution $u$ of the following 2D stochastic Navier--Stokes equations on $\bbS$:
\begin{align}
du - \left[ \nu \Delta' u - (u \cdot \nablaS) u  - \nablaS p \right]dt = f dt + G\,dW(t) & \quad\text{ in } \bbS \times (0,T), \label{eq:5.6}\\
\ddiv' u = 0 & \quad \text{ in } \bbS \times (0,T),\label{eq:5.7} \\
u = 0 & \quad \mbox{ on } \partial Q \times (0,T), \label{eq:5.8}\\
u(0, \cdot ) = u_0 & \quad\text{ in } \bbS,\label{eq:5.9}
\end{align}
where $u=(u_1, u_2)$ and $\DDelta$, $\nablaS$, $\ddivS$ are the differential operators for $\R^2$-valued vector field. Assumptions on initial data and external forcing will be specified later (see Assumptions \ref{ass5.1},~\ref{ass_sphere}). Here, $G : \Rp \to \mathcal{T}_2(\R^N; \rH)$ and $W(t)$, $t \ge 0$ is an $\R^N$-valued Wiener process such that
\begin{equation}
\label{eq:wiener}
G(t)dW(t) := \sum_{j=1}^N g^j(t) d\beta_j(t),
\end{equation}
where $N \in \N$, $\left(\beta_j\right)_{j=1}^N$ are i.i.d real valued Brownian motions as before and $\left\{g^j\right\}_{j=1}^N$ are elements of $\rH$, with certain relation to $\tge^j$, which is specified later in Assumption~\ref{ass_sphere}.

Now, we specify assumptions on the initial data $\widetilde{u}_0^\eps$ and external forcing $\widetilde{f}_\eps$, $\widetilde{g}_\eps^j$.
\begin{assumption}\label{ass5.1}
Let $\left(\Omega, \mathcal{F}, \mathbb{F}, \mathbb{P}\right)$ be the given filtered probability space. Let us assume that $p \ge 2$ and that  $\tu_0^\eps \in \rH_\eps$, for $\eps \in (0,1]$,   such that for some $C_1 > 0$
\begin{equation}
\label{eq:5.10}
\|\tu_0^\eps\|_{\LQeps} = C_1\eps^{1/2}, \qquad \eps \in (0,1].
\end{equation}
We also assume that  $\tfe \in L^p([0,T]; \rV_\eps^\ast)$, for $\eps \in (0,1]$,  such that for some $C_2 > 0$,
\begin{equation}
\label{eq:5.11}
\int_0^T\|\tfe(s)\|^p_{\rV_\eps^\ast}\,ds \le C_2 \eps^{p/2}, \qquad \eps \in (0,1].
\end{equation}
Let $\wtd{W}^\eps$ be an $\R^N$-valued Wiener process as before and assume that
\[
\widetilde{G}_\eps \in L^p(0,T; \mathcal{T}_2(\R^N; \rH_\eps)), \;\;\mbox{for $\eps \in (0,1]$,}
\]
such that, using convention \eqref{eq:5.5}, for each $j=1,\ldots,N$, $t \in [0,T]$
\begin{equation}\label{eq:5.12}
\|\tge^j(t)\|^p_{\LQeps} = O(\eps^{p/2}),\;\; \eps \in (0,1].
\end{equation}
\end{assumption}

Projecting the stochastic NSE (defined on thin domain $Q_\eps$) \eqref{eq:5.1}--\eqref{eq:5.4} onto $\rH_\eps$ using the Leray-Helmholtz projection operator and using the definitions of operators from \S\,\ref{section-prelim}, we obtain the following abstract It\^o equation in $\rH_\eps$, $t \in [0,T]$
\begin{equation}
\label{eq:5.13}
d \tue (t) + \left[\nu \rA_\eps \tue(t) + B_\eps(\tue(t), \tue(t))\right]dt = \tfe(t)\,dt + \widetilde{G}_\eps(t)\,d \tWe(t), \quad \tue(0) = \tu_0^\eps.
\end{equation}

\begin{definition}
\label{defn5.2}
Let $\eps \in (0,1]$. A martingale solution to \eqref{eq:5.13} is a system
\[\left(\Omega, \calF, \bbF, \bbP, \wtd{W}_\eps, \wtd{u}_\eps\right)\]
where $\left(\Omega, \calF, \bbP \right)$ is a probability space and $\bbF = \left(\calF_t\right)_{t \ge 0}$ is a filtration on it, such that
\begin{itemize}
\item $\wtd{W}_\eps$ is a $\R^N$-valued Wiener process on $\left(\Omega, \calF, \bbF, \mathbb{P}\right)$,
\item $\wtd{u}_\eps$ is $\rV_\eps$-valued progressively measurable process, $\rH_\eps$-valued weakly continuous $\bbF$-adapted process such that\footnote{The space $X^w$ denotes a topological space $X$ with weak topology. In particular, $C([0,T]; X^w)$ is the space of weakly continuous functions $\rv:[0,T] \to X$.} $\bbP$-a.s.
\[
    \tue(\cdot, \omega) \in C([0,T],{\rm H}_{\eps}^w) \cap L^2(0,T; \rV_\eps),
\]
\[\bbE \left[\frac{1}{2} \sup_{0\le s \le T} \|\tue(s) \|^2_{\LQeps}
  + \nu \int_0^T \| \nabla \tue (s) \|^2_{\LQeps}\,ds    \right] < \infty\]
and, for all $t \in [0,T]$ and $\rv \in \rV_\eps$, $\bbP$-a.s.,

\begin{align}
\label{eq:5.14}
& (\tue(t),\rv)_{\LQeps} +  \nu \int_0^t (\nabla \tue(s), \nabla \rv)_{\LQeps} \,ds + \int_0^t \left\langle B_\eps( \tue(s), \tue(s)), \rv\right\rangle_\eps\,ds \\
&\quad = (\tu_0^\eps, \rv)_{\LQeps} + \int_0^t \left\langle\tfe(s),\rv\right\rangle_{\eps}\,ds + \left( \int_0^t \widetilde{G}_\eps(s)\,d\wtd{W}_\eps(s),\rv \right)_\LQeps . \nn
\end{align}
\end{itemize}
\end{definition}

In the following remark we show that a martingale solution $\tue$ of \eqref{eq:5.13}, as defined above, satisfies an equivalent equation in the weak form.
\begin{remark}
\label{rem_mod_equation}
Let $\tue = \tue(t)$, $t \ge 0$ be a martingale solution of \eqref{eq:5.13}. We will use the following notations
\begin{equation}
    \label{eq:notation_stoch}
    \tale(t) := \tMe[\tue (t)], \qquad \tble(t) := \tNe[\tue(t)], \qquad \ale(t) := \oMe[\tue(t)],
\end{equation}
and also from Lemma~\ref{lem.avg_op_vec_prop} we have
\[\tue(t) = \tale(t) + \tble(t), \qquad t \in [0,T].\]
Then, for $\varphi \in \rU$, and the map $R_\eps$ defined in \eqref{eqn-4_17}--\eqref{eqn-4_18} we have
\begin{align*}
\left(\tale , \test \right)_{\LQeps} & = \eps \left(\ale, \varphi \right)_{\LS} ,\\
\left(\tMe u_0^\eps, \test \right)_{\LQeps} & = \eps \left(\oMe u_0^\eps, \varphi \right)_{\LS}, \\
\left(\nabla \tale , \nabla \test \right)_{\LQeps} & = \eps \left(\nablaS \ale, \nablaS \varphi \right)_{\LS}, \\
\left\langle \tMe \left(\left[\tale \cdot \nabla \right]\tale\right), \test \right\rangle_\eps & = \eps\left\langle\left[\ale \cdot \nablaS\right]\ale, \varphi \right\rangle,\\
\left\langle\tMe \tfe, \test \right\rangle_\eps & = \eps \left\langle\oMe \tfe, \varphi \right\rangle,\\
\left(\int_0^t\tMe \left[\wtd{G}_\eps(s)\,d\wtd{W}_\eps(s)\right], \test \right)_\LQeps & = \eps \left(\int_0^t \oMe \left[ \wtd{G}_\eps(s)\,d\wtd{W}_\eps(s)\right], \varphi \right)_\LS,
\end{align*}
and using Lemma~\ref{lemma_avgN} and Lemma~\ref{lem.avg_op_vec_prop}, we can rewrite the weak formulation identity \eqref{eq:5.14} as follows.
\begin{align}
\label{eq:snse_mod}
\left(\alpha_\eps(t), \varphi \right)_{\LS} & =  \left(\oMe \tu_0^\eps, \varphi \right)_{\LS} - \nu \int_0^t \left(\nablaS \alpha_\eps(s), \nablaS \varphi \right)_{\LS} ds \\
&\quad  -  \int_0^{t} \left\langle\left[\alpha_\eps(s) \cdot \nablaS \right]\alpha_\eps(s), \varphi \right\rangle ds - \frac{1}{\eps} \int_0^{t} \left\langle\left[\tble(s) \cdot \nabla \right]\tble(s), \test \right\rangle_\eps ds \nn \\
&\quad + \int_0^{t} \left\langle\oMe \tfe(s), \varphi \right\rangle ds  + \left(\int_0^t \oMe \left[ \wtd{G}_\eps(s)\,d\wtd{W}_\eps(s)\right], \varphi \right)_\LS. \nn
\end{align}
\end{remark}

Next, we present the definition of a martingale solution for 2D stochastic NSE on $\bbS$.
\begin{definition}
\label{defn_mart_SNSE_sphere}
A {martingale solution} to equation \eqref{eq:5.6}--\eqref{eq:5.9} is a system
\[\left(\widehat{\Omega}, \widehat{\calF}, \widehat{\bbF}, \widehat{\bbP}, \widehat{W}, \widehat{u}\right)\]
where $\left(\widehat{\Omega}, \widehat{\calF},\widehat{\bbP} \right)$ is a probability space and $\widehat{\bbF} = \left(\widehat{\calF}_t\right)_{t \ge 0}$ is a filtration on it, such that
\begin{itemize}
\item $\what{W}$ is an $\R^N$-valued Wiener process on $\left(\what{\Omega}, \what{\calF}, \what{\bbF}, \what{\bbP}\right)$,
\item $\what{u}$ is $\rV$-valued progressively measurable process, $\rH$-valued continuous $\what{\bbF}$-adapted process such that
\[
    \widehat{u}(\cdot, \omega) \in C([0,T],{\rm H}) \cap L^2(0,T; \rV),
\]
\[\what{\bbE} \left[ \sup_{0\le s \le T} \|\hu(s) \|^2_{\LS}
  + \nu \int_0^T \| \nablaS \hu (s) \|^2_{\LS}\,ds    \right] < \infty\]
and
\begin{align}
\label{eq:mart_sol_sphere}
\left(\hu(t), \varphi\right)_\LS + \nu \int_0^t \left(\nablaS \hu(s), \nablaS \varphi \right)_\LS\,ds + \int_0^{t} \left\langle\left[\hu(s) \cdot \nablaS \right]\hu(s), \varphi \right\rangle\,ds  \\
 = \left(u_0, \varphi \right)_\LS + \int_0^{t} \left\langle f(s), \varphi \right\rangle\,ds + \left(\int_0^t G(s) d \widehat{W}(s), \varphi \right)_\LS, \nn
\end{align}
for all $t \in [0,T]$ and $\varphi \in \rV$.
\end{itemize}
\end{definition}

\begin{assumption}\label{ass_sphere}
Let $p \ge 2$. Let $\left(\widehat{\Omega}, \widehat{\mathcal{F}}, \widehat{\mathbb{F}}, \widehat{\mathbb{P}}\right)$ be the given probability space, $u_0 \in \rH$ such that
\begin{equation}
\label{eq:initial_data_convg}
\lim_{\eps\rightarrow 0} \oMe\widetilde{u}_0^\eps = u_0 \quad \text{weakly in } \rH.
\end{equation}
Let $f \in L^p([0,T]; \rV^\ast)$, such that for every
$s \in [0,T]$,
\begin{equation}
\label{eq:convg_ext_force}
\lim_{\eps\rightarrow 0} \langle \oMe \tfe(s), \rv \rangle
  = \langle f(s), \rv \rangle \text{ for all } \rv \in \rV.
\end{equation}
And finally, we assume that $G \in L^p(0,T; \mathcal{T}_2(\R^N; \rH))$, such that for each $j=1,\ldots,N$ and $s \in [0,T]$, $\oMe \tge^j(s)$ converges weakly to $g^j(s)$ in $\mathbb{L}^2(\bbS)$ as $\eps \to 0$ and
\begin{equation}\label{eq:bdd_coeff}
\int_0^T\|g^j(t)\|^2_{\LS}dt \le M,
\end{equation}
for some $M > 0$.
\end{assumption}

\begin{remark}[{\bf Existence of martingale solutions}]
\label{rem_mart_sol_SNSE_shell}
For a fixed $\eps > 0$, the existence of a martingale solution to \eqref{eq:5.1}--\eqref{eq:5.2} with Dirichlet boundary conditions was shown in \cite{[FG95]}. More recently, the first named author along with Motyl proved the existence of martingale solutions to \eqref{eq:5.1}--\eqref{eq:5.2} on possibly unbounded domain with transport type noise; see \cite{[BM13]}. In \cite{[BS20]}, the first named author and Slavik have established the existence of martingale solutions for much more complex so-called primitive equations in three dimensions but with the boundary conditions \eqref{eq:5.3}--\eqref{eq:5.3a}. Thus, we believe that one can establish the existence of martingale solutions to \eqref{eq:5.1}--\eqref{eq:5.4} by following the standard approaches as in \cite{[FG95], [BM13], [BS20]}.

Moreover, in \cite{[BM13]}, the existence of a unique strong solution to 2D stochastic NSEs, \eqref{eq:5.6}--\eqref{eq:5.9}, was proved by showing the existence of a martingale solution and proving the pathwise uniqueness of such a martingale solution.
\end{remark}

We end this subsection by stating the main theorem of this article.

\begin{theorem}
\label{thm:main_thm}
Let the given data $\widetilde{u}_0^\eps$, $u_0$, $\tfe$, $f$, $\tge^j$, $g^j$, $j \in \{1, \cdots, N\}$ satisfy Assumptions~\ref{ass5.1} and \ref{ass_sphere}. Let $\left( \Omega, \mathcal{F}, \mathbb{F}, \mathbb{P}, \wtd{W}_\eps, \tue\right)$
be a martingale solution of \eqref{eq:5.1}--\eqref{eq:5.4} as defined in Definition~\ref{defn5.2}. Then, the averages in the vertical direction of this martingale solution i.e. $\hae := \oMe\tue$ converge to {the unique martingale solution}, $\left(\widehat{\Omega}, \widehat{\mathcal{F}}, \widehat{\mathbb{F}}, \widehat{\mathbb{P}}, \widehat{W}, \what{u}\right)$, of \eqref{eq:5.6}--\eqref{eq:5.9} in $L^2(\widehat{\Omega}\times[0,T]\times \bbS)$.
\end{theorem}

\begin{remark}
\label{rem.one_prob_space}
According to Remark~\ref{rem_mart_sol_SNSE_shell}, for every $\eps \in [0,1]$ there exists a martingale solution of \eqref{eq:5.1}--\eqref{eq:5.4} as defined in Definition~\ref{defn5.2}, i.e. we will obtain a tuple $\left(\Omega_\eps, \mathcal{F}_\eps, \mathbb{F}_\eps, \mathbb{P}_\eps, \wtd{W}_\eps, \tue\right)$ as a martingale solution. It was shown in \cite{[Jakubowski97]} that is enough to consider only one probability space, namely,
\[\left(\Omega_\eps, \mathcal{F}_\eps, \mathbb{P}_\eps\right) = \left([0,1], \mathcal{B}([0,1]), \mathcal{L}\right) \qquad \forall\, \eps \in (0,1],\]
where $\mathcal{L}$ denotes the Lebesgue measure on $[0,1]$. Thus, it is justified to consider the probability space $\left( \Omega, \mathcal{F}, \mathbb{P}\right)$ independent of $\eps$ in Theorem~\ref{thm:main_thm}.
\end{remark}

\begin{remark}\label{rem-Kuksin}
Chueshov and Kuksin in \cite{[CK08a]} stated a result; see \cite[Proposition~12]{[CK08a]}, which is similar to our Theorem~\ref{thm:main_thm}. These  authors study the so called Leray-$\alpha$ approximation \cite{[Leray34]}, of the stochastic 3D Navier--Stokes equations
\begin{equation}
    \label{eq:leray_NSE}
    u_t + \nu A_\eps u + B_\eps(G_\alpha u, u) = f_\eps + \dot{W_\eps}
\end{equation}
where $A_\eps$, $B_\eps$ are the Stokes operator and bilinear form, as introduced in \S\,\ref{section-prelim}, $G_\alpha = \left(1 + \alpha A_\eps\right)^{-1}$, and $\dot{W_\eps}$ is the white noise.
 They proved the existence and the uniqueness of a solutions $u^\alpha_\eps$ to  \eqref{eq:leray_NSE} and
 they  established the existence of a unique stationary measure $\mu_\eps^\alpha$ to \eqref{eq:leray_NSE} for  sufficiently  non-degenerate  noise, and proved that every solution of \eqref{eq:leray_NSE} converges to it in law  distribution as time goes to infinity. The main result of their work was to show the convergence of $M_\eps \mu^\alpha_\eps$ to the unique stationary measure $\mu$ of 2D stochastic Navier--Stokes equations as $\alpha, \eps \to 0$ simultaneously \cite[Theorem~10]{[CK08a]}, see \S\,\ref{sec:Average} for the definition of the map $M_\eps$ (which can be generalised for measures). In the statement of the Proposition~12, the authors state that under suitable growth and convergence assumptions on the initial data $u_\eps^0$, the law of  $M_\eps u^\alpha_\eps$ converges to the law of $u$ as $\eps, \alpha \to 0$ simultaneously, in appropriate functional space.

{In retrospect, it seems that our result has striking similarities to that of the statement  in \cite[Proposition~12]{[CK08a]}. At the same time there are subtle differences, which we state in the following. Firstly, we consider Dirichlet boundary conditions at the level of flat domain ($Q$, see \eqref{eqn-snse4}) in contrast to ``periodic" boundary conditions. Secondly, we consider the original 3D stochastic Navier--Stokes equations instead of the Leray-$\alpha$ approximation, and in particular, show that for every martingale solution $\tue$ of \eqref{eqn-snse1}--\eqref{eqn-snse5}, there exists a subsequence of $\left\{\oMe \tue\right\}_{\eps > 0}$ which converges to a martingale solution of the 2D stochastic Navier--Stokes equations in a suitable functional space. Finally, we provide a self-contained proof with all the details of our main result along with some higher order moment estimates, see Lemma~\ref{lemma_higher_estimates_u}.}
\end{remark}

\subsection{Estimates}
\label{sec:estimates_thin_film}

From this point onward we will assume that for every $\eps \in (0,1]$ there exists a martingale solution $\left(\Omega, \calF, \bbF, \bbP, \wtd{W}_\eps, \wtd{u}_\eps\right)$ of \eqref{eq:5.13}.
Please note that we do not claim neither we use the uniqueness of this solution. \\
The main aim of this subsection is to obtain estimates for $\alpha_\eps$ and $\widetilde{\beta}_\eps$ uniform in $\eps$ using the estimates for the process $\tue$.

The energy inequality \eqref{s_energy} and the higher-order estimates \eqref{eq:hoe1}--\eqref{eq:hoe2}, satisfied by the process $\tue$, as obtained in Lemma~\ref{lemma_apriori_sto_tu} and Lemma~\ref{lemma_higher_estimates_u} is actually a consequence (essential by-product) of the existence proof. In principle, one obtains these estimates (uniform in the approximation parameter $N$) for the finite-dimensional process $\tue^{(N)}$ (using Galerkin approximation) with the help of the It\^o lemma. Then, using the lower semi-continuity of norms, convergence result ($\tue^{(N)} \to \tue$ in some sense),  one can establish the estimates for the limiting process. Such a methodology was employed in a proof of Theorem 4.8 in the recent paper \cite{[BMO17]} by the first named author, Motyl and Ondrej\'at.

In Lemma~\ref{lemma_apriori_sto_tu} and Lemma~\ref{lemma_higher_estimates_u} we present a formal proof where we assume that one can apply (ignoring the existence of Lebesgue and stochastic integrals) the It\^o lemma to the infinite dimensional process $\tue$. The idea is to showcase (though standard) the techniques involved in establishing such estimates.

\begin{lemma}
\label{lemma_apriori_sto_tu}
Let ${\tu}_0^\eps \in \rH_\eps$, $\tfe \in L^2\left([0,T]; \rV^\ast_\eps\right)$ and $\tge \in L^2\left([0,T]; \mathcal{T}_2(\R^N; \rH_\eps\right)$. Then, the martingale solution $\wtd{u}_\eps$
of \eqref{eq:5.13} satisfies the following energy inequality
\begin{equation}\label{s_energy}
\begin{aligned}
&\bbE \left[\frac{1}{2} \sup_{0\le s \le T} \|\wtd{u}_\eps(s) \|^2_{\lL^2(Q_\eps)} ds
  + \nu \int_0^T \| \nabla \wtd{u}_\eps (s) \|^2_{\lL^2(Q_\eps)}    \right] \\
&\qquad \le  \|u^\eps_0\|^2_{\lL^2(Q_\eps)}
+ \frac{1}{\nu} \int_0^T \|\tfe(s)\|^2_{\rV^\ast_\eps}
+ K \int_0^T \sum_{j=1}^N \|\wtd{g}^\eps_j(s)\|^2_{\lL^2(Q_\eps)} ds,
\end{aligned}
\end{equation}
where $K$ is some positive constant independent of $\eps$.
\end{lemma}

\begin{proof}
Let $\tue$ be the solution of \eqref{eq:5.13}. Using the It\^{o} formula for the function $\|\xi\|^2_{\lL^2(Q_\eps)}$ with the process $\wtd{u}_\eps$, for a fixed $t \in [0,T]$ we have
\begin{equation}\label{eqn-Ito}
\begin{aligned}
&\|\wtd{u}_{\eps}(t)\|^2_{\LQeps} + 2\nu \int_0^t \|\nabla \wtd{u}_\eps(s)\|^2_{\lL^2(Q_\eps)} ds =
\|\tu_0^\eps\|^2_{\LQeps} + 2\int_0^t \langle \tfe(s), \tue(s) \rangle_\eps ds
 \\
&\qquad \quad  + 2 \left(\int_0^t \wtd{G}_\eps(s)d\wtd{W}_\eps(s), \wtd{u}_\eps(s)\right)_{\LQeps}
 + \int_0^t \|\wtd{G}_\eps(s)\|^2_{\mathcal{T}_2(\R^N; \rH_\eps)} ds.
\end{aligned}
\end{equation}
Using the Cauchy-Schwarz and the Young inequality for the third term on the right-hand-side of the above inequality, we obtain the following estimate
\[
|\langle\tfe, \wtd{u}_\eps\rangle_\eps | \le \|\wtd{u}_\eps\|_{\rV_\eps} \|\tfe\|_{\rV^\ast_\eps}
                     \le \frac{\nu}{2}  \|\wtd{u}_\eps\|^2_{\rV_\eps}  + \frac{1}{2\nu} \|\tfe\|^2_{\rV^\ast_\eps}
\]
which we use in \eqref{eqn-Ito} to simplify \eqref{eqn-Ito} 
\begin{equation}\label{eqn-Ito2}
\begin{aligned}
&\|\wtd{u}_{\eps}(t)\|^2_{\LQeps} + \nu \int_0^t \|\nabla \wtd{u}_\eps(s)\|^2_{\lL^2(Q_\eps)} ds \le
\|\tu_0^\eps\|^2_{\LQeps} + \frac{1}{\nu}\int_0^t \|\tfe(s) \|^2_{\rV^\ast_\eps} ds  \\
&\qquad \quad + 2 \left(\int_0^t \wtd{G}_\eps(s)d\wtd{W}_\eps(s), \wtd{u}_\eps(s)\right)_{\LQeps}
 + \int_0^t \sum_{j=1}^N \|\wtd{g}^\eps_j(s)\|^2_{\LQeps} ds.
\end{aligned}
\end{equation}
Using the Burkholder-Davis-Gundy inequality, we have
\begin{equation}\label{BDG}
\begin{aligned}
&\bbE  \sup_{0\le t \le T} \left| \left(\int_0^t \wtd{G}_\eps(s)d\wtd{W}_\eps(s), \wtd{u}_\eps(s)\right)_{\LQeps} \right| \\
&\qquad
\le C \bbE \left( \int_0^T \|\wtd{u}_\eps(s)\|^2_{\LQeps} \|\wtd{G}_\eps(s)\|^2_{\mathcal{T}_2(\R^N;\rH_\eps)} ds\right)^{1/2} \\
& \qquad \le C \bbE \left[ \left(\sup_{0\le t \le T} \|\wtd{u}_\eps(t)\|^2_{\LQeps} \right)^{1/2}
   \left( \int_0^T \sum_{j=1}^N \|\wtd{g}_j^\eps(s)\|^2_{\LQeps} ds\right)^{1/2} \right] \\
&\qquad \le \frac{1}{4} \bbE\left( \sup_{0 \le t \le T} \| \wtd{u}_\eps(t)\|^2_{\LQeps} \right)
 + C \int_0^T\sum_{j=1}^N \|\wtd{g}_j^\eps(s)\|^2_{\LQeps} ds .
\end{aligned}
\end{equation}
Taking the supremum of \eqref{eqn-Ito2} over the inteveral $[0,T]$, then taking
expectation and using inequality \eqref{BDG} we infer the energy inequality \eqref{s_energy}.
\end{proof}

Let us recall the following notations, which we introduced earlier, for $t \in [0,T]$
\begin{equation}
\label{eq:notation_1}
\wtd{\alpha}_\eps(t) := \wtd{M}_\eps [\wtd{u}_\eps(t)], \qquad \wtd{\beta}_\eps := \wtd{N}_\eps[ \wtd{u}_\eps(t)], \qquad \alpha_\eps := \ocirc{M}_\eps[\wtd{u}_\eps(t)].
\end{equation}

\begin{lemma}\label{lem: a priori M sto}
 Let $\wtd{u}_\eps$ be a martingale solution of \eqref{eq:5.13} and the Assumption~\ref{ass5.1} hold, in particular for $p = 2$. Then
\begin{equation}\label{a priori M sto}
\begin{aligned}
&\bbE \left[ \frac{1}{2} \sup_{0\le t\le T} \|\alpha_\eps(t)\|^2_{\lL^2(Q)} +
\nu \int_0^t \| \nabla' \alpha_\eps(s)\|^2_{\lL^2(Q)} ds \right] \le C_1^2 + \frac{C_2}{\nu} + C_3T.
\end{aligned}
\end{equation}
where $C_1$, $C_2$ are positive constants from \eqref{eq:5.10} and \eqref{eq:5.11} respectively and $C_3 > 0$ (determined within the proof) is another constant independent of $\eps$.
\end{lemma}

\begin{proof}
Let $\tue$ be a martingale solution of \eqref{eq:5.13}, then it satisfies the energy inequality \eqref{s_energy}. From \eqref{eqn-pythagoras1} and \eqref{eqn-pythagoras2},
we have $\|\wtd{\alpha}_\eps(t)\|^2_{\lL^2(Q_\eps)} \le \|\wtd{u}_\eps(t)\|^2_{\lL^2(Q_\eps)}$, $\|\nabla \wtd{\alpha}_\eps(t)\|^2_{\lL^2(Q_\eps)} \le \|\nabla \wtd{u}_\eps(t)\|^2_{\lL^2(Q_\eps)}$ respectively.
Thus, from these bounds and energy inequality \eqref{s_energy}, we get
\begin{equation}\label{Menergy}
\begin{aligned}
&\bbE \left[\frac{1}{2} \sup_{0\le t \le T}\|\wtd{\alpha}_\eps(t)\|^2_{\lL^2(Q_\eps)}
+ \nu \int_0^T \| \nabla \wtd{\alpha}_\eps(s)\|^2_{\lL^2(Q_\eps)} ds \right] \\
& \qquad  \le \|\tu^\eps_0\|^2_{\lL^2(Q_\eps)}
+ \frac{1}{\nu} \int_0^T \|\tfe(s)\|^2_{\rV^\ast_\eps} ds +
  K  \int_0^T \sum_{j=1}^N \|\wtd{g}^\eps_j(s)\|^2_{\lL^2(Q_\eps)} ds.
\end{aligned}
\end{equation}
Now using the scaling property \eqref{eqn-vscaling} we rewrite the left hand side of \eqref{Menergy} as
\begin{equation}\label{eq:above}
\begin{aligned}
&\bbE\left[\frac{1}{2}\sup_{0 \le t\le T}\eps\|\alpha_\eps(t)\|^2_{\lL^2(Q)} +
\eps \nu \int_0^T \|\nabla' \alpha_\eps(s)\|^2_{\lL^2(Q)} ds\right] \\
&\qquad \le \E \left[\|u_0\|^2_{\lL^2(Q_\eps)}
+ \frac{1}{\nu} \int_0^T \|\tfe(s)\|^2_{\rV^\ast_\eps} ds
+ \frac{C}{2} \sum_{j=1}^N \int_0^t \|\wtd{g}^\eps_j(s)\|^2_{\lL^2(Q_\eps)} ds \right].
\end{aligned}
\end{equation}
By the assumptions on $\wtd{g}_j^\eps$ \eqref{eq:5.12}, there exists a positive constant $c$ such that for every $1 \le j \le N$, $t \in [0,T]$,
\begin{equation}
\label{eq:assump_g_c}
\|\wtd{g}_j^\eps(t)\|^2_{\lL^2(Q_\eps)} \le c \eps.
\end{equation}
Therefore, using \eqref{eq:5.10}, \eqref{eq:5.11} and \eqref{eq:assump_g_c} in \eqref{eq:above}, cancelling $\eps$ on both sides and defining $C_3 = NKc$, we infer inequality \eqref{a priori M sto}.
\end{proof}

From the results of Lemma~\ref{lem: a priori M sto} and Lemma~\ref{lemma:M N}, we deduce that
\begin{equation}
\label{eq:alpha_bdd}
\left\{ \alpha_\eps \right\}_{\eps > 0}\, \mbox{ is bounded in }\,
L^2(\Omega;L^\infty(0,T; \rH) \cap L^2(0,T; \rV)).
\end{equation}
Since $\rV$ can be embedded into $\mathbb{L}^6(\bbS)$, by using interpolation between $L^\infty(0,T; \rH)$ and $L^2(0,T; \mathbb{L}^6(\bbS))$ we obtain
\begin{equation}
\label{eq:5.28}
\bbE \int_0^T \|\ale(s)\|^2_{\mathbb{L}^3(\bbS)}\,ds \le C\;.
\end{equation}

\begin{lemma}\label{a priori N sto}
Let $\wtd{u}_\eps$ be a martingale solution of \eqref{eq:5.13} and Assumption~\ref{ass5.1} hold, in particular for $p =2$. Then
\begin{equation}
\label{eq:5.29}
\bbE \left[\frac{1}{2}\sup_{0 \le t \le T} \|\wtd{\beta}_\eps(t)\|^2_{\lL^2(Q_\eps)}
+ \nu \int_0^T \| \nabla \wtd{\beta}_\eps(s) \|_{\lL^2(Q_\eps)}^2 ds \right]
   \le \eps\left(C_1^2 + \frac{C_2}{\nu} + C_3T\right).
\end{equation}
\end{lemma}

\begin{proof}
Let $\tue$ be a martingale solution of \eqref{eq:5.13}, then it satisfies the energy inequality \eqref{s_energy}. From \eqref{eqn-pythagoras1} and \eqref{eqn-pythagoras2},
we have $\|\wtd{\beta}_\eps(t)\|^2_{\lL^2(Q_\eps)} \le \| \wtd{u}_\eps(t)\|^2_{\lL^2(Q_\eps)}$
and $\|\nabla \wtd{\beta}_\eps(t)\|^2_{\lL^2(Q_\eps)} \le \|\nabla \wtd{u}_\eps(t) \|^2_{\lL^2(Q_\eps)}$.
Thus, using the above estimates, Assumption~\eqref{ass5.1}, \eqref{eq:assump_g_c} in the energy inequality \eqref{s_energy}, we infer \eqref{eq:5.29}.
\end{proof}

In the following lemma we obtain some higher order estimates (on a formal level) for the martingale solution $\tue$, which will be used to obtain the higher order estimates for the processes $\ale$ and $\tble$, which will be further used later to prove the main result, Theorem~\ref{thm:main_thm} of this article.

\begin{lemma}
\label{lemma_higher_estimates_u}
Let Assumption~\ref{ass5.1} hold true and $\tue$ be a martingale solution of \eqref{eq:5.13}. Then, for $p > 2$ we have following estimates
\begin{equation}
\label{eq:hoe1}
\bbE\sup_{0\le s\le T} \|\tue(s)\|^p_{\LQeps}
 \le C_2\left(p, \widetilde{u}_0^\eps, \tfe,\widetilde{G}_\eps\right)\exp\left(K_p\, T\right)
\end{equation}
and
\begin{equation}
\label{eq:hoe2}
\E \int_0^T \|\tue(t)\|_{\LQeps}^{p-2}\|\tue(t)\|_{\rV_\eps}^2dt \le C_2\left(p, \widetilde{u}_0^\eps, \tfe,\widetilde{G}_\eps\right) \left[1 + K_{p}\,T\exp\left(K_{p}\,T\right) \right],
\end{equation}
where
\[
C_2\left(p, \widetilde{u}_0^\eps, \tfe, \widetilde{G}_\eps \right):=
\|\widetilde{u}_0^\eps\|^p_{\LQeps} + \nu^{-p/2}\|\tfe\|^p_{L^p(0,T; \rV'_\eps)}
   + \left(\frac{1}{4}p^2 (p-1) + \frac{K_1^2}{p}\right) \|\widetilde{G}_\eps\|^p_{L^p(0,T;\calT_2)},
\]
\[K_{p} := \left(\frac{K_1^2}{p}+p\right)\frac{(p-2)}{2},\]
and $K_1$ is a constant from the Burkholder-Davis-Gundy inequality.
\end{lemma}

\begin{proof}
Let $F(x) = \|x\|^p_{\LQeps}$ then
\[
  \frac{\partial F}{\partial x} = \nabla F
   = p \|x\|^{p-2}_{\LQeps} x,
\]
and
\begin{equation}\label{eqn:2}
 \left| \frac{\partial^2 F}{\partial x^2} \right|
  \le p(p-1) \|x\|^{p-2}_{\LQeps}.
\end{equation}
Applying the It\^{o} lemma with $F(x)$ and process $\widetilde{u}_\eps$ for
$t \in [0,T]$, we have
\begin{align*}
& \|\tue(t)\|^p_{\LQeps} = \|\tue(0)\|^p_{\LQeps}
  + p \int_0^t \|\tue(s)\|_{\LQeps}^{p-2}
\langle -\nu \rA_\eps \tue(s) - B_\eps(\tue(s),\tue(s))
+ \widetilde{f}_\eps(s), \tue(s) \rangle_\eps\, ds \\
&\quad + p \int_0^t \|\tue(s)\|^{p-2}_{\LQeps} \left( \tue(s), \widetilde{G}_\eps(s) d\widetilde{W}_\eps(s) \right)_{\LQeps}
+ \frac{1}{2}\int_0^t \mathrm{Tr}
\left( \frac{\partial^2 F}{\partial x^2} (\widetilde{G}_\eps(s),\widetilde{G}_\eps(s))\right)\,ds .
\end{align*}
Using the fact that
$\langle B_\eps(\tue,\tue), \tue \rangle_\eps = 0$ and
$ \langle A_\eps \tue,\tue \rangle_\eps = \|\nabla \tue\|_{\LQeps}^2 $ we arrive at
\begin{align*}
&\|\tue(t)\|^p_{\LQeps} = \|\tue(0)\|^p_{\LQeps}
 - p\nu \int_0^t \|\tue(s)\|_{\LQeps}^{p-2} \| \nabla \tue(s) \|_{\LQeps}^2
 + p \int_0^t \|\tue(s)\|_{\LQeps}^{p-2} \langle \widetilde{f}_\eps(s), \tue(s) \rangle_\eps\, ds \\
&\quad + p \int_0^t \|\tue(s)\|^{p-2}_{H_\eps}
 \left( \tue(s), \widetilde{G}_\eps(s) d\widetilde{W}_\eps(s)
\right)_{\LQeps}
+ \frac{1}{2}
\int_0^t \mathrm{Tr}
\left( \frac{\partial^2 F}{\partial x^2} (\widetilde{G}_\eps(s),\widetilde{G}_\eps(s))\right)ds.
\end{align*}
Using \eqref{eqn:2} and the Cauchy-Schwarz inequality, we get
\begin{align*}
&\|\tue(t)\|^p_{\LQeps}  + p \nu \int_0^t \|\tue(s)\|_{\LQeps}^{p-2} \|\nabla \tue(s)\|_{\LQeps}^2 ds\\
&\;\; \le \|\tue(0)\|^p_{\LQeps}
 + p \int_0^t \|\tue(s)\|_{\LQeps}^{p-2} \|\widetilde{f}_\eps(s)\|_{\rV'_\eps} \|\tue(s)\|_{\rV_\eps}\,ds \\
&\quad + p \int_0^t \|\tue(s)\|_{\LQeps}^{p-2} \left( \tue(s), \widetilde{G}_\eps(s) d\widetilde{W}_\eps(s) \right)_\LQeps + \frac{p(p-1)}{2}
\int_0^t \|\tue(s)\|_{\LQeps}^{p-2} \|\widetilde{G}_\eps(s)\|^2_{\calT_2(\R^N; \rH_\eps)}\, ds,
\end{align*}
where we recall
\[
\|\widetilde{G}_\eps(s)\|^2_{\calT_2(\R^N; \rH_\eps)} = \sum_{j=1}^N \|\tge^j(s)\|^2_{\LQeps}.
\]
Using the generalised Young inequality $abc \le a^q/q + b^r/r + c^s/s$ (where
$1/q+1/r+1/s=1$) with $a = \sqrt{\nu} \|\tue\|_{\LQeps}^{p/2-1} \|\tue\|_{\rV_\eps}$,
$b=\|\tue\|_{\LQeps}^{p/2-1}$, $c = \frac{1}{\sqrt{\nu}}\|f_\eps\|_{\rV'_\eps}$ and
exponents $q=2, r=p, s=2p/(p-2)$ we get
\begin{equation}\label{equ:3}
 \nu\|\tue\|_{\LQeps}^{p-2}
 \|f_\eps\|_{\rV'_\eps}
\|\tue\|_{\rV_\eps}
\le \frac{\nu}{2} \|\tue\|_{\LQeps}^{p-2} \|\tue\|_{\rV_\eps}^2
 + \frac{1}{p\nu^{p/2}} \|\widetilde{f}_\eps\|^p_{\rV'_\eps} + \frac{p-2}{2p} \|\tue\|_{\LQeps}^p.
\end{equation}
Again using the Young inequality with exponents $p/(p-2)$, $p/2$ we get
\begin{equation}\label{equ:4}
\|\tue\|_{\LQeps}^{p-2}\|\widetilde{G}
_\eps\|^2_{\calT_2(\R^N; \rH_\eps)}
\le \frac{p-2}{p} \|\tue\|_{\LQeps}^p + \frac{p}{2} \|\widetilde{G}_\eps\|^p_{\calT_2(\R^N; \rH_\eps)}.
\end{equation}
Using \eqref{equ:3} and \eqref{equ:4} and recalling $\|\tue\|_{\rV_\eps} = \|\nabla \tue\|_\LQeps$, we obtain
\begin{equation}\label{equ:5}
\begin{aligned}
&\|\tue(t)\|^p_{\LQeps}
 + \frac{p \nu}{2}
\int_0^t \|\tue(s)\|_{\LQeps}^{p-2} \|\nabla \tue(s)\|_{\LQeps}^2 ds
 \\
& \le \|\tue(0)\|^p_{\LQeps}
+ \frac{p(p-2)}{2}
\int_0^t \|\tue(s)\|_{\LQeps}^p\,ds
 + \nu^{-p/2} \int_0^t \|\widetilde{f}_\eps(s)\|^p_{\rV'_\eps}\, ds
\\
&\; + \frac{1}{4}p^2(p-1) \int_0^t \|\widetilde{G}_\eps(s)\|^p_{\calT_2(\R^N; \rH_\eps)}\, ds  + p \int_0^t \|\tue(s)\|_{\LQeps}^{p-2}
 \left( \tue(s),\widetilde{G}_\eps(s) d\widetilde{W}_\eps(s) \right)_\LQeps.
\end{aligned}
\end{equation}
Since $\tue$ is a martingale solution of \eqref{eq:5.13} it satisfies the energy inequality \eqref{s_energy}, hence the real-valued random variable
\[
\mu_\eps(t) = \int_0^t \|\tue(s)\|_{\LQeps}^{p-2} \left( \tue(s), \widetilde{G}_\eps(s) d\widetilde{W}_\eps(s) \right)_\LQeps
\]
is a $\mathcal{F}_t$-martingale. Taking expectation both sides of \eqref{equ:5}
we obtain
\begin{equation}
\begin{aligned}
& \bbE \|\tue(t)\|^p_{\LQeps}
+ \frac{p \nu}{2}
\bbE \int_0^t \|\tue(s)\|_{\LQeps}^{p-2} \|\nabla \tue(s)\|_{\LQeps}^2 ds \\
& \qquad \le
\|\tue(0)\|^p_{\LQeps}
+ \frac{p(p-2)}{2} \E \int_0^t \|\tue(s)\|_{\LQeps}^p ds
+ \nu^{-p/2} \int_0^t \|\widetilde{f}_\eps(s)\|^p_{\rV'_\eps}\,ds \\
& \qquad \quad + \frac{1}{4}p^2 (p-1) \int_0^t \|\widetilde{G}_\eps(s)\|^p_{\calT_2(\R^N; \rH_\eps)}\,ds.
\end{aligned}
\end{equation}
Therefore, by the Gronwall lemma we obtain
\[
\bbE\|\tue(t)\|^p_{\LQeps} \le C\left(\widetilde{u}_0^\eps, \tfe,\widetilde{G}_\eps\right)\exp\left( p\frac{(p-2)t}{2}\right),
\]
where
\[
C\left(\widetilde{u}_0^\eps, \tfe,\widetilde{G}_\eps\right) := \|\widetilde{u}_0^\eps\|_{\LQeps}^p
+ \nu^{-p/2} \|f\|^p_{L^p(0,T; \rV'_\eps)}
  + \frac{1}{4} p^2(p-1)
\| \widetilde{G}_\eps\|^p_{L^p(0,T;\calT_2(\R^N; \rH_\eps))} .
\]
By Burkholder-Davis-Gundy inequality, we have
\begin{align}\label{equ:9}
\bbE & \left(\sup_{0\le s \le t}\left| \int_0^s \|\tue(\sigma)\|^{p-2}_{\LQeps}
  \left( \tue(\sigma), \widetilde{G}_\eps(\sigma) d\widetilde{W}_\eps(\sigma) \right)_\LQeps \right|\right) \nonumber \\
& \le K_1 \bbE \left( \int_0^t \|\tue(s)\|_{\LQeps}^{2p-4}
 \|\tue(s)\|^2_{\LQeps} \|\widetilde{G}_\eps(s)\|^2_{\calT_2(\R^N; \rH_\eps)}\,ds \right)^{1/2} \nn \\
&\le K_1 \bbE \left[ \sup_{0\le s\le t} \|\tue(s)\|_{\LQeps}^{p/2}
 \left(\int_0^t \|\tue(s)\|_{\LQeps}^{p-2} \|\widetilde{G}_\eps(s)\|^2_{\calT_2(\R^N; \rH_\eps)} ds  \right)^{1/2}
 \right] \nn\\
&\le
 \frac{1}{2} \bbE \sup_{0\le s \le t} \|\tue(s)\|_{\LQeps}^p
+ \frac{K^2_1}{2} \bbE \int_0^t \|\tue(s)\|_{\LQeps}^{p-2} \|\widetilde{G}_\eps(s)\|^2_{\calT_2(\R^N; \rH_\eps)}\, ds \nn\\
&\le
 \frac{1}{2} \bbE \sup_{0\le s \le t} \|\tue(s)\|^p_{\LQeps}
+ \frac{K_1^2}{2} \frac{(p-2)}{p} \bbE \int_0^t \|\tue(s)\|_{\LQeps}^p ds +
\frac{K_1^2}{p} \int_0^t \|\widetilde{G}_\eps(s)\|^p_{\calT_2(\R^N; \rH_\eps)}\,ds
\end{align}
where in the last step we have used the Young inequality with exponents $p/(p-2)$
and $p/2$.

\noindent Taking supremum over $0\le s\le t$ in \eqref{equ:5}
and using \eqref{equ:9} we get
\begin{align}
\label{equ:10}
\frac{1}{2} \bbE & \sup_{0 \le s\le t} \|\tue(s)\|^p_{\LQeps}
+ \frac{\nu}{p} \bbE \int_0^t \|\tue(s)\|_{\LQeps}^{p-2} \|\nabla \tue(s)\|^2_{\LQeps} ds \\
&\le
\|\tue(0)\|^p_{\LQeps}
+ \left( \frac{K_1^2}{p} + p \right) \frac{(p-2)}{2}
\int_0^t \bbE \sup_{0 \le s \le \sigma} \|\tue(s)\|_{\LQeps}^p d\sigma  \nn \\
&\qquad + \nu^{-p/2} \int_0^t \|\widetilde{f}_\eps(s)\|^p_{\rV'_\eps} ds + \left(\frac{1}{4}p^2(p-1)
+ \frac{K_1^2}{p}\right) \int_0^t \|\widetilde{G}_\eps(s)\|^p_{\calT_2(\R^N; \rH_\eps)}\,ds. \nn
\end{align}
Thus using the Gronwall lemma, we obtain
\[
\bbE\sup_{0\le s\le t} \|\tue(s)\|^p_{\LQeps}
 \le C_2\left(p, \widetilde{u}_0^\eps, \tfe,\widetilde{G}_\eps\right)\exp\left(K_p\,t\right),
\]
where $C_2\left(p, \widetilde{u}_0^\eps, \tfe,\widetilde{G}_\eps\right)$ and $K_p$ are the constants as defined in the statement of lemma. We deduce \eqref{eq:hoe2} from \eqref{equ:10} and \eqref{eq:hoe1}.
\end{proof}

In the following lemma we will use the estimates from previous lemma to obtain higher order estimates for $\alpha_\eps$ and $\tble$.

\begin{lemma}
Let $p > 2$. Let $\tue$ be a martingale solution of \eqref{eq:5.13} and Assumption~\ref{ass5.1} hold with the chosen $p$. Then, the processes $\ale$ and $\tble$ (as defined in \eqref{eq:notation_1}) satisfy the following estimates
\begin{equation}
\label{eq:hoe_alpha}
\E\sup_{t\in[0,T]}\|\ale(t)\|^p_{\LS} \le K(\nu, p)\exp\left(K_p\,T\right),
\end{equation}
and
\begin{equation}
\label{eq:hoe_beta}
\E\sup_{t\in[0,T]}\|\tble(t)\|^p_{\LQeps} \le \eps^{p/2}K(\nu, p)\exp\left(K_p\,T\right),
\end{equation}
where $K(\nu,p)$ is a positive constant independent of $\eps$ and $K_p$ is defined in Lemma~\ref{lemma_higher_estimates_u}.
\end{lemma}
\begin{proof}
The lemma can be proved following the steps of Lemma~\ref{lem: a priori M sto} and Lemma~\ref{a priori N sto} with the use of Lemma~\ref{lem.avg_op_vec_prop}, scaling property from Lemma~\ref{lem.vec_scaling}, the Cauchy-Schwarz inequality, Assumptions~\ref{ass5.1}, \ref{ass_sphere} and the estimates obtained in Lemma~\ref{lemma_higher_estimates_u}.
\end{proof}


\subsection{Tightness}
\label{sec:tightness_thin_film}
In this subsection we will prove that the family of laws induced by the processes $\alpha_\eps$ is tight on an appropriately chosen topological space $\mathcal{Z}_T$. In order to do so we will consider the following functional spaces for fixed $T > 0$:

\begin{itemize}
    \item $C([0,T];\rU^\ast)$ is the space of continuous functions $u:[0,T] \rightarrow \rU^\ast$ with the topology
$\mathbf{T}_1$ induced by the norm $\|u\|_{C[0,T];\rU^\ast} = \sup_{t\in [0,T]} \|u(t)\|_{\rU^\ast}$.

\item $L^2_{\mathrm{w}}(0,T; \rV)$ is the space $L^2(0,T; \rV)$ with the weak topology $\mathbf{T}_2$,

\item $L^2(0,T; \rH)$ is the space of measurable functions $u : [0,T] \to \mathrm{H}$ such that
\[
\|u\|_{L^2(0,T; \mathrm{H})} = \left(\int_0^T \| u(t)\|_\rH^2\,dt \right)^{\frac12} < \infty,
\]
with the topology $\mathbf{T}_3$ induced by the norm $\|u\|_{L^2(0,T; \mathrm{H})}$.

\item Let $\rH_w$ denote the Hilbert space $\rH$ endowed with the weak topology.

$C([0,T]; \mathrm{H}_w)$ is the space of weakly continuous functions $u: [0,T] \to \mathrm{H}$ endowed with the weakest topology $\mathbf{T}_4$ such that for all $h \in \mathrm{H}$ the mappings
\[
C([0,T]; \mathrm{H}_w) \ni u \to \left( u(\cdot), h \right)_{\mathrm{H}} \in C([0,T]; \R)
\]
are continuous. In particular $u_n \to u$ in $C([0,T]; \rH_w)$ iff for all $h \in \rH \colon$
\[\lim_{n \to \infty} \sup_{t \in [0,T]} \left|\left( u_n(t) - u(t), h \right)_\rH \right| = 0.\]
\end{itemize}
Let
\begin{equation}
\label{def:ZT}
\mathcal{Z}_T = C([0,T]; \mathrm{U}^\ast) \cap L^2_{\mathrm{w}}(0,T; \rV) \cap L^2(0,T; \rH) \cap C([0,T]; \mathrm{H}_w),
\end{equation}
and $\mathcal{T}$ be the supremum\footnote{$\mathcal{T}$ is the supremum of topologies $\mathbf{T}_1$, $\mathbf{T}_2$, $\mathbf{T}_3$ and $\mathbf{T}_4$, i.e. it is the coarsest topology on $\mathcal{Z}_T$ that is finer than each of $\mathbf{T}_1$, $\mathbf{T}_2$, $\mathbf{T}_3$ and $\mathbf{T}_4$} of the corresponding topologies.

\begin{lemma}
\label{lem Lip sto}
The set of measures $\left\{\mathcal{L}(\alpha_\eps),\, \eps \in (0,1] \right\}$ is tight on $\left(\mathcal{Z}_T, \mathcal{T}\right)$.
\end{lemma}

\begin{proof}
Let $\wtd{u}_\eps(t)$, for some fixed $\eps > 0$ be a martingale solution of \eqref{eq:5.13}. Then by Remark~\ref{rem_mod_equation} for $\wtd{\alpha}_\eps$, $\wtd{\beta}_\eps$ and $\alpha_\eps$ as in \eqref{eq:notation_1}, we have for $\varphi \in \rU$, $t \in [0,T]$, $\mathbb{P}$-a.s.
\begin{align}
\label{eqn-start}
&\left(\wtd{\alpha}_\eps(t) - \wtd{\alpha}_\eps (s), \test\right)_\LQeps  = -\nu \int_s^t \left(\nabla \wtd{\alpha}_\eps(r), \nabla \test \right)_\LQeps dr
    \nonumber \\
  &\qquad - \int_s^t \left\langle \wtd{M}_\eps \left[ \wtd{\alpha}_\eps(r) \cdot \nabla \wtd{\alpha}_\eps(r) \right], \test \right\rangle_\eps dr - \int_s^t \left\langle\wtd{M}_\eps \left[\wtd{\beta}_\eps(r) \cdot \nabla \wtd{\beta}_\eps(r) \right], \test \right\rangle_\eps dr \nn \\
     &\qquad + \int_s^t \left\langle\wtd{M}_\eps \tfe(r), \test \right\rangle_\eps dr
     + \left(\int_s^t \wtd{M}_\eps\left[\wtd{G}_\eps(r)\,d\wtd{W}_\eps(r)\right], \test\right)_\LQeps\nonumber \\
  &\qquad =: I_1 + I_2 + \ldots + I_5.
\end{align}
The proof of lemma turns out to be  a direct application of Corollary~\ref{corA.2.2}. Indeed, by  Lemma~\ref{lem: a priori M sto}, assumptions  $(a)$ and $(b)$ of Corollary~\ref{corA.2.2} are satisfied and therefore,
it is sufficient to show that the sequence $\left(\alpha_\eps\right)_{\eps > 0}$ satisfies the Aldous condition $[\textbf{A}]$, see Definition~\ref{defnA.1.8},  in space $\rU^\ast$.

Let $\theta \in (0,T)$ and $\left(\tau_\eps\right)_{\eps > 0}$ be a sequence of stopping times such that $0 \le \tau_\eps \le \tau_\eps +\theta\le T$. We start by estimating each term in the right-hand-side of \eqref{eqn-start}. We will use the H\"older inequality, the scaling property from Lemma~\ref{lem.vec_scaling}, the Poincar\'{e} type inequality \eqref{NuV}, the Ladyzhenskaya type inequality from Corollary~\ref{cor-Lady-anisotropic}, identities from Lemma~\ref{lem:Reps_Sob} and the a priori estimates from Lemmas~\ref{lem: a priori M sto}, \ref{a priori N sto}.

In what follows, we will prove that each of the five process from equality \eqref{eqn-start} satisfies the Aldous condition $[\textbf{A}]$.
In order to help the reader, we will divide the following part of the proof into five parts.
\begin{itemize}
\item For the first term we obtain
\begin{align}
\label{est for I1}
\bbE \left|I_1^\eps(\tau_\eps + \theta) - I_1^\eps(\tau_\eps)\right| & = \nu \bbE \left|\int_{\tau_\eps}^{\tau_\eps+\theta} \left(\nabla \tale(s) , \nabla \test \right)_\LQeps\,ds\right|\nn \\
& = \nu \eps \bbE\left| \int_{\tau_\eps}^{{\tau_\eps}+\theta} \left(\nablaS \ale(s) , \nablaS \varphi \right)_\LS\right|\,ds \nn\\
& \le \nu\eps \|\nabla'\varphi\|_{\lL^2(Q)}
\E  \int_{\tau_\eps}^{\tau_\eps + \theta} \|\nabla'\alpha_\eps(s)\|_{\lL^2(Q)} ds \nn \\
 & \le \nu \eps \|\varphi\|_\rV \left(\bbE \int_{\tau_\eps}^{\tau_\eps + \theta}\|\nabla^\prime \alpha_\eps(s)\|^2_{\lL^2(Q)} ds\right)^{1/2} \theta^{1/2} \nn\\
& \le \eps \nu C\|\varphi\|_{\rV}\,\theta^{1/2} := c_1 \eps \cdot \theta^{1/2} \|\varphi\|_{\rV} .
\end{align}

\item Now we consider the first nonlinear term
\begin{align}
    \label{est for I2}
    \bbE & \left|I_2^\eps(\tau_\eps + \theta) - I_2^\eps(\tau_\eps)\right| = \bbE \left|\int_{\tau_\eps}^{{\tau_\eps}+\theta} \left\langle\left[\tale(s) \cdot \nabla \right]\tale(s), \test \right\rangle_\eps\,ds\right| \nn\\
    & = \bbE \left|\int_{\tau_\eps}^{{\tau_\eps}+\theta} \left(-\int_{Q_\eps}\left[\tale(s,\bx)\cdot\nabla\right][\test](\bx)\cdot \tale(s,\bx)\,d\bx\right)\,ds\right| \nn\\
& \le  \bbE \int_{\tau_\eps}^{{\tau_\eps}+\theta} \|\tale(s)\|_\LQeps \|\tale(s)\|_{\mathbb{L}^3(Q_\eps)}\left\|\nabla\test\right\|_{\mathbb{L}^6(Q_\eps)}\,ds \nn \\
& = \eps \bbE \int_{\tau_\eps}^{{\tau_\eps}+\theta}\|\ale(s)\|_\LS\|\ale(s)\|_{\mathbb{L}^3(\bbS)}\|\nablaS\varphi\|_{\lL^6(\bbS)}\,ds \nn \\
& \le \eps \left[\bbE \left(\sup_{s \in [0,T]} \|\ale(s)\|^2_\LS \right) \right]^{1/2} \left[\bbE \|\ale\|^2_{L^2(0,T; \mathbb{L}^3(\bbS))} \right]^{1/2}\|\varphi\|_{\mathbb{W}^{1,6}(\bbS)} \theta^{1/2} \nn \\
&\le \eps C \|\varphi\|_{\rU}\,\theta^{1/2} := c_2 \eps \cdot \theta^{1/2} \|\varphi\|_{\rU} .
\end{align}

\item Similarly for the other nonlinear term, we get
\begin{align}
    \label{est for I3}
    \bbE & \left|I_3^\eps(\tau_\eps + \theta) - I_3^\eps(\tau_\eps)\right| = \bbE \left|\int_{\tau_\eps}^{{\tau_\eps}+\theta} \left\langle\left[\tble(s) \cdot \nabla \right]\tble(s), \test \right\rangle_\eps\,ds\right| \nn \\
& = \bbE \left|\int_{\tau_\eps}^{{\tau_\eps}+\theta} \left(-\int_{Q_\eps}\left[\tble(s,\bx)\cdot\nabla\right][\test](\bx)\cdot \tble(s,\bx)\,d\bx\right)\,ds\right| \nn\\
&\le \bbE \int_{\tau_\eps}^{{\tau_\eps}+\theta} \|\tble(s)\|_\LQeps \|\tble(s)\|_{\mathbb{L}^3(Q_\eps)} \left\|\nabla\test\right\|_{\mathbb{L}^6(Q_\eps)}\,ds \nn \\
&\le \eps^{2/3} C \bbE \int_{\tau_\eps}^{{\tau_\eps}+\theta}\|\tble(s)\|_\LQeps\|\tble(s)\|_{\rV_\eps}\|\nablaS\varphi\|_{\lL^6(\bbS)}\,ds \nn \\
&\le \eps^{2/3} C \left[\bbE \left(\sup_{s\in[0,T]}\|\tble(s)\|^2_\LQeps\right)\right]^{1/2} \left[\bbE\|\tble\|^2_{L^2(0,T; \rV_\eps)}\right]^{1/2} \|\varphi\|_{\mathbb{W}^{1,6}(\bbS)}\,\theta^{1/2} \nn \\
&\le \eps^{5/3} C \|\varphi\|_{\rU}\,\theta^{1/2} =: c_3 \eps^{5/3} \cdot \theta^{1/2} \|\varphi\|_{\rU}.
\end{align}

\item Now for the term corresponding to the external forcing $\tfe$, we have using the independence of $\tMe \tfe$ on $x_3$ and assumption \eqref{eq:5.11}
\begin{align}
\label{est for I4}
\bbE  \left|I_4^\eps(\tau_\eps + \theta) - I_4^\eps(\tau_\eps)\right|& = \bbE \left|\int_{\tau_\eps}^{{\tau_\eps}+\theta} \left\langle\tMe \tfe(s), \test \right\rangle_\eps\,ds\right| \nn \\
&= \bbE \left|\int_{\tau_\eps}^{{\tau_\eps}+\theta} \left(\int_\bbS \inteps  [M_\eps \tfe](s, \bx)\varphi(\bx')\,d\bx'\,dx_3\right)\,ds\right| \nn \\
& = \eps \bbE\left|\int_{\tau_\eps}^{{\tau_\eps}+\theta} \left(\int_\bbS [\ocirc{M}_\eps \tfe](s,\bx') \varphi(\bx')\,d\bx'\right)\,ds\right|\nn \\
& \le \eps \bbE \int_{\tau_\eps}^{{\tau_\eps}+\theta} \|\ocirc{M}_\eps \tfe(s)\|_{\rV^\ast}\|\varphi\|_\rV\,ds \nn \\
& \le \sqrt{\eps} C \left[\bbE\|\tfe\|^2_{L^2(0,T; \rV_\eps^\ast)} \right]^{1/2}\|\varphi\|_{\rU}\,\theta^{1/2} \nn\\
& \le \eps C \|\varphi\|_{\rU}\,\theta^{1/2}  =: c_4 \eps \cdot \theta^{1/2} \|\varphi\|_{\rU}.
\end{align}

\item
At the very end  we are left to deal with the last term corresponding to the stochastic forcing. Using the independence of $M_\eps\tge^j$ from $x_3$, It\^o isometry, scaling (see Lemma~\ref{lem.vec_scaling}) and assumption \eqref{eq:5.12}, we get
\begin{align}
\label{est for I5}
& \bbE  \left|I_5^\eps(\tau_\eps + \theta) - I_5^\eps(\tau_\eps)\right|^2  = \bbE \left| \left( \int_{\tau_\eps}^{{\tau_\eps}+\theta} \tMe \left[ \widetilde{G}_\eps(s)\,d \tWe(s)\right], \test \right)_\LQeps\right|^2 \nn \\
& \qquad = \bbE \left| \int_{\tau_\eps}^{{\tau_\eps}+\theta} \left( \int_\bbS \inteps \sum_{j=1}^N M_\eps \left[\tge^j(s,\bx)\,d\beta_j(s)\right]\varphi(\bx')\,dx_3\,d\bx'\right)\right|^2 \nn\\
&\qquad= \eps \bbE \left|\int_{\tau_\eps}^{{\tau_\eps}+\theta} \sum_{j=1}^N \int_\bbS \ocirc{M}_\eps \left[\tge^j(s, \bx')d\beta_j(s)\right] \varphi(\bx')\,d\bx'\right|^2 \nn \\
&\qquad\le \eps \bbE  \left\|\sum_{j=1}^N \int_{\tau_\eps}^{{\tau_\eps}+\theta} \ocirc{M}_\eps \left[\tge^j(s)d\beta_j(s)\right]\right\|_\LS^2 \|\varphi\|_\LS^2 \nn \\
&\qquad=  \E \left( \int_{\tau_\eps}^{{\tau_\eps}+\theta} \sum_{j=1}^N \eps\|\ocirc{M}_\eps \tge^j(s)\|_\LS^2\,ds \right) \|\varphi\|^2_\LS \nn \\
&\qquad=  \E \left(\int_{\tau_\eps}^{{\tau_\eps}+\theta} \sum_{j=1}^N \|\tMe\tge^j(s)\|^2_\LQeps ds \right)\|\varphi\|_\LS^2 \nn \\
&\qquad \le  \E  \left(\int_{\tau_\eps}^{{\tau_\eps}+\theta} \sum_{j=1}^N \|\tge^j(s)\|^2_\LQeps\,ds\right)\|\varphi\|_\LS^2 \nn \\
&\qquad \le \eps Nc \|\phi\|_{\rU}^2 \theta := c_5 \eps \cdot \theta \|\phi\|_{\rU}^2 .
\end{align}
\end{itemize}
After having proved what we had promised, we are ready to conclude the proof of Lemma \ref{lem Lip sto}. Since for every $t > 0$
\[ \left(\tale (t), \test \right)_\LQeps = \eps \left(\ale(t), \varphi \right)_\LS,\]
one has for $\varphi \in \rU$,
\begin{align}
\label{eq:aldous_2}
\left\|\ale(\tau_\eps + \theta) - \ale(\tau_\eps)\right\|_{\rU^\ast} & = \sup_{\|\varphi\|_{\rU} = 1} \left(\ale(\tau_\eps + \theta) - \ale(\tau_\eps), \varphi \right)_\LS \\
& = \frac{1}{\eps} \sup_{\|\phi\|_{\rU} = 1} \left(\tale(\tau_\eps + \theta) - \tale(\tau_\eps), \test \right)_\LQeps. \nn
\end{align}
Let us fix $\kappa > 0$ and $\gamma > 0$. By equality \eqref{eqn-start}, the sigma additivity property of probability measure and \eqref{eq:aldous_2}, we have
\begin{align*}
\mathbb{P} \left(\left\{\|\ale(\tau_\eps + \theta) - \ale(\tau_\eps)\|_{{\rU^\ast}} \ge \kappa \right\}\right) \le \frac{1}{\eps} \sum_{i=1}^5 \mathbb{P} \left(\left\{\sup_{\|\phi\|_{\rU} = 1} \left|J_i^\eps(\tau_\eps + \theta) - J_i^\eps(\tau_\eps)\right| \ge \kappa\right\}\right).
\end{align*}
Using the Chebyshev's inequality, we get
\begin{align}
\label{eq:aldous_3}
\mathbb{P} \left(\left\{\|\ale(\tau_\eps + \theta) - \ale(\tau_\eps)\|_{{\rU^\ast}} \ge \kappa \right\}\right) & \le \frac{1}{\kappa \eps} \sum_{i=1}^4 \bbE \left(\sup_{\|\phi\|_{\rU} = 1} \left|J_i^\eps(\tau_\eps + \theta) - J_i^\eps(\tau_\eps)\right| \right) \\
& \;\; + \frac{1}{\kappa^2 \eps} \bbE \left(\sup_{\|\phi\|_{\rU} = 1} \left|J_5^\eps(\tau_\eps + \theta) - J_5^\eps(\tau_\eps)\right|^2\right) \nn.
\end{align}
Thus, using estimates \eqref{est for I1}--\eqref{est for I5} in \eqref{eq:aldous_3}, we get
\begin{align}
\label{eq:aldous_4}
\mathbb{P}  \left(\left\{\|\ale(\tau_\eps + \theta) - \ale(\tau_\eps)\|_{{\rU^\ast}} \ge \kappa \right\}\right) \le \frac{1}{\kappa \eps} \eps \theta^{1/2} \left[c_1 + c_2 + c_3 \eps^{2/3} + c_4\right] + \frac{1}{\kappa^2 \eps} c_5 \eps \theta.
\end{align}
Let $\delta_i = \left(\dfrac{\kappa}{5 c_i} \gamma\right)^2$, for $i = 1, \cdots, 4$ and $\delta_5 = \dfrac{\kappa^2}{5 c_5} \gamma$. Choose $ \delta = \max_{i \in \{1, \cdots, 5\}}\delta_i$. Hence,
\[ \sup_{\eps > 0} \sup_{0 \leq \theta \leq \delta} \mathbb{P}  \left(\left\{\|\ale(\tau_\eps + \theta) - \ale(\tau_\eps)\|_{{\rU^\ast}} \ge \kappa \right\}\right) \leq \gamma.\]
Since $\ale$ satisfies the Aldous condition $[\textbf{A}]$ in ${\rU^\ast}$, we conclude the proof of Lemma~\ref{lem Lip sto} by invoking Corollary~\ref{corA.2.2}.
\end{proof}

\subsection{Proof of Theorem~\ref{thm:main_thm}}
\label{sec:proof-theorem}
For every $\eps > 0$, let us define the following intersection of spaces
\begin{equation}
\label{eq:y_eps space}
\mathcal{Y}_T^\eps = L^2_w(0,T; \rV_\eps) \cap C ([0,T]; \rH_\eps^w).
\end{equation}
Now, choose a countable subsequence $\left\{\eps_k\right\}_{k \in \N}$ converging to $0$. For this subsequence define a product space $\mathcal{Y}_T$ given by

\begin{equation}
\label{eqn-T_T space}
\mathcal{Y}_T = \Pi_{k \in \N} \mathcal{Y}_T^{\eps_k}\;,
\end{equation}
and a function $\eta \colon \Omega \to \mathcal{Y}_T$ by
\[\eta(\omega) = \left(\wtd{\beta}_{\eps_1}(\omega), \wtd{\beta}_{\eps_2}(\omega), \cdots, \right) \in \mathcal{Y}_T\;.\]
Now with this $\mathcal{Y}_T$-valued function we define a $\mathcal{Y}_T$-sequence
\[\eta_k \equiv \eta\,, \quad k \in \N\;.\]

\noindent Then by Lemma~\ref{lem Lip sto} and the definition of sequence $\eta_k$, the set of measures $\left\{\mathcal{L}\left(\alpha_{\eps_k}, \eta_k\right), k \in \N\right\}$ is tight on $\mathcal{Z}_T \times \mathcal{Y}_T$.

Thus, by the Jakubowski-Skorohod theorem\footnote{The space $\mathcal{Z}_T \times \mathcal{Y}_T \times C([0,T]; \R^N)$ satisfies the assumption of Theorem~\ref{thmA.1.4}. Indeed, since $\mathcal{Z}_T$ and $\mathcal{Y}_T^\eps$, $\eps > 0$ satisfies the assumptions (see \cite[Lemma~4.10]{[BD18]}) and $C([0,T]; \R^N)$ is a Polish space and thus automatically satisfying the required assumptions.} there exists a subsequence $\left(k_n\right)_{n \in \N}$, a probability space $(\what{\Omega}, \what{\mathcal{F}}, \hp)$ and, on this probability space, $\mathcal{Z}_T \times \mathcal{Y}_T \times C([0,T]; \R^N)$-valued random variables
 $(\what{u}, \what{\eta}, \widehat{W})$, $\left(\what{\alpha}_{\eps_{k_n}}, \what{\eta}_{k_n}, \widehat{W}_{\eps_{k_n}}\right), n \in \N$ such that
\begin{equation}
\label{eq:5.44}
\begin{split}
& \left(\what{\alpha}_{\eps_{k_n}}, \what{\eta}_{k_n}, \widehat{W}_{\eps_{k_n}}\right) \,\mbox{ has the same law as }\, \left({\alpha}_{\eps_{k_n}}, \eta_{k_n}, \widetilde{W}\right)\\
&\mbox{ on } \calB\left(\mathcal{Z}_T \times \mathcal{Y}_T \times C([0,T]; \R^N)\right)
\end{split}
\end{equation}
and
\begin{equation}
\label{eq:5.44a}
\left(\what{\alpha}_{\eps_{k_n}}, \what{\eta}_{k_n}, \widehat{W}_{\eps_{k_n}}\right) \to \left(\what{u}, \what{\eta}, \widehat{W} \right) \,\mbox{ in }\,\mathcal{Z}_T \times \mathcal{Y}_T \times C(0,T; \R^N),\quad \hp\mbox{-a.s.}
\end{equation}
In particular, using marginal laws, and definition of the process $\eta_k$, we have
\begin{equation}
    \label{eq:law_alp_bet}
    \mathcal{L}\left(\what{\alpha}_{\eps_{k_n}}, \what{\beta}_{\eps_{k_n}}\right) = \mathcal{L}\left(\alpha_{\eps_{k_n}}, \widetilde{\beta}_{\eps_{k_n}}\right) \; \mbox{ on } \calB\left(\mathcal{Z}_T \times \mathcal{Y}_T^{{\eps_{k_n}}}\right)
\end{equation}
where $\what{\beta}_{\eps_{k_n}}$ is the $k_n$th component of $\mathcal{Y}_T$-valued random variable $\what{\eta}_{k_n}$. We are not interested in the limiting process $\widehat{\eta}$ and hence will not discuss it further.

Using the equivalence of law of $\widehat{W}_{\eps_{k_n}}$ and $\widetilde{W}$ on $C([0,T];\R^N)$ for $n \in \N$ one can show that $\widehat{W}$ and $\widehat{W}_{\eps_{k_n}}$ are $\R^N$-valued Wiener processes (see \cite[Lemma~5.2 and Proof]{[BGJ13]} for details).

The convergence $\what{\alpha}_{\eps_{k_n}} \to \what{u}\,\mbox{ in }\,\mathcal{Z}_T$ (see \eqref{def:ZT}), $\hp\mbox{-a.s.}$ precisely means that
\begin{align*}
\what{\alpha}_{\eps_{k_n}} &\to \hu \, \mbox{ in }\,C([0,T];\rU^\ast)\,, \qquad \qquad
\what{\alpha}_{\eps_{k_n}} \rightharpoonup \hu \, \mbox{ in }\,L^2(0,T; \rV)\;,\\
\what{\alpha}_{\eps_{k_n}} &\to \hu \, \mbox{ in }\,L^2(0,T; \rH)\,, \qquad \qquad \; \;
\what{\alpha}_{\eps_{k_n}} \to \hu \, \mbox{ in }\,C([0,T]; \rH_w)\;,
\end{align*}
and
\[\widehat{W}_{\eps_{k_n}} \to \widehat{W}\, \mbox{in }\,C([0,T]; \R^N)\;.\]
Let us denote the subsequence $(\what{\alpha}_{\eps_{k_n}}, \what{\beta}_{\eps_{k_n}}, \widehat{W}_{\eps_{k_n}})$ again by $(\hae, \hbe, \widehat{W}_\eps)_{\eps \in (0,1]}$.\\
Note that since $\calB(\mathcal{Z}_T \times \mathcal{Y}_T \times C([0,T]; \R^N) \subset \calB(\mathcal{Z}_T) \times \mathcal{B}(\mathcal{Y}_T)\times \calB(C([0,T]; \R^N))$, the functions $\hu$, $\what{\eta}$ are $\mathcal{Z}_T$, $\mathcal{Y}_T$ Borel random variables respectively.

The following auxiliary result which is needed in the proof of Theorem~\ref{thm:main_thm}, cannot be deduced directly from the Kuratowski Theorem (see Theorem~\ref{thmB.1.1}).

\begin{lemma}
\label{lemma5.7}
Let $T > 0$ and $\mathcal{Z}_T$ be as defined in \eqref{def:ZT}. Then the following sets $C([0,T];\rH) \cap \mathcal{Z}_T$, $L^2(0,T; \rV) \cap \mathcal{Z}_T$ are Borel subsets of $\mathcal{Z}_T$.
\end{lemma}
\begin{proof}
See Appendix~\ref{s:AppB.2}.
\end{proof}

By Lemma~\ref{lemma5.7}, $C([0,T]; \rH) \cap \calZ_T$ is a Borel subset of $\mathcal{Z}_T$. Since $\ale \in C([0,T]; \rH)$, $\hp$-a.s. and $\hae$, $\ale$ have the same laws on $\mathcal{Z}_T$, thus
\begin{equation}
\label{eq:5.45}
\mathcal{L}(\hae)\left(C([0,T]; \rH)\right) = 1\,, \quad \eps > 0\;,
\end{equation}
and from estimate \eqref{a priori M sto} and \eqref{eq:hoe_alpha}, for $p \in[2, \infty)$
\begin{equation}
\label{eq:5.46}
\sup_{\eps > 0} \what{\E} \left(\sup_{0 \le s \le T} \|\hae(s)\|^p_{\LS} \right) \le K_1(p).
\end{equation}
Since $L^2(0,T; \mathrm{V}) \cap \mathcal{Z}_T$ is a Borel subset of $\calZ_T$ (Lemma~\ref{lemma5.7}), $\ale$ and $\hae$ have same laws on $\calZ_T$; from \eqref{a priori M sto}, we have
\begin{equation}
\label{eq:5.47}
\sup_{\eps > 0} \what{\E} \left[ \int_0^T \|\nabla^\prime \hae(s)\|^2_{\LS}\,ds \right] \le K_2.
\end{equation}

Since the laws of $\eta_{k_n}$ and $\widehat{\eta}_{k_n}$ are equal on $\mathcal{Y}_T$,  we infer that  the corresponding marginal laws are also equal.  In other words, the laws on $\calB\left(\mathcal{Y}_T^{\eps_{k_n}}\right)$ of  $\mathcal{L}(\what{\beta}_{\eps_{k_n}})$ and $\mathcal{L}(\wtd{\beta}_{\eps_{k_n}})$ are equal for every ${k_n}$.

Therefore, from the estimates \eqref{eq:5.29} and \eqref{eq:hoe_beta} we infer for $p \in [2, \infty)$
\begin{equation}
\label{eq:beta_estimate_1}
\widehat{\E} \left(\sup_{0 \le s \le T} \|\hbe(s)\|^{p}_{\LQeps} \right) \le K_3(p) \eps^{p/2}, \qquad \eps \in (0,1]
\end{equation}
and
\begin{equation}
\label{eq:beta_estimate_2}
\widehat{\E} \left[ \int_0^T \|\nabla \hbe(s)\|^2_{\LQeps}\,ds \right] \le K_4 \eps, \qquad \eps \in (0,1].
\end{equation}

By inequality \eqref{eq:5.47} we infer that the sequence $(\hae)_{\eps > 0}$ contains a subsequence, still denoted by $(\hae)_{\eps > 0}$ convergent weakly (along the sequence $\eps_{k_n}$) in the space $L^2(\homega \times [0,T]; \rV)$. Since $\hae \to \hu$ in $\mathcal{Z}_T$ $\hp$-a.s., we conclude that $\hu \in L^2(\homega \times [0,T]; \rV)$, i.e.
\begin{equation}
\label{eq:5.48}
\what{\E} \left[\int_0^T \|\nabla^\prime \hu(s)\|^2_{\LS}\,ds \right] < \infty.
\end{equation}
Similarly by inequality \eqref{eq:5.46}, for every $p \in [2, \infty)$ we can choose a subsequence of $(\hae)_{\eps>0}$ convergent weak star (along the sequence $\eps_{k_n}$) in the space $L^p(\homega; L^\infty(0,T; \rH))$ and, using \eqref{eq:5.46}, we infer that
\begin{equation}
\label{eq:5.49}
\what{\E} \left(\sup_{0 \le s \le T} \|\hu(s)\|^{p}_{\LS} \right) < \infty.
\end{equation}

Using the convergence from \eqref{eq:5.44a} and estimates \eqref{eq:5.46}--\eqref{eq:5.49} we will prove certain term-by-term convergences which will be used later to prove Theorem~\ref{thm:main_thm}.

In order to simplify the notation, in the result below we write $\lim_{\eps \to 0}$ but we mean $\lim_{k_n\to \infty}$.

Before stating the next lemma, we recall the functional space $\Psi \colon$
\begin{equation}
    \label{eq:Psi}
    \Psi:= \left\{ \psi \in C_0^\infty(Q;\R^2) : \ddivS = 0 \mbox{ in } \bbS\right\}.
\end{equation}

\begin{lemma}
\label{lemma5.8}
For all $t \in [0,T]$ and $\varphi \in \Psi$, we have (along the sequence $\eps_{k_n}$)
\begin{itemize}
\item[(a)] $\lim_{\eps \to 0} \hE \left[\int_0^T \left|\left(\hae(t) - \hu(t),\varphi \right)_\LS\right|\,dt\right] = 0$,
\item[(b)] $\lim_{\eps \to 0} \left(\hae(0) - u_0,\varphi \right)_\LS = 0$,
\item[(c)] $\lim_{\eps \to 0} \hE\left[\int_0^T \left| \int_0^t \nu \left( \nablaS \hae(s) - \nablaS \hu(s), \nablaS\varphi \right)_\LS ds\right|\,dt\right] = 0$,
\item[(d)] $\lim_{\eps \to 0} \hE \left[ \int_0^T \left| \int_0^t \left\langle\left[\hae(s) \cdot \nabla^\prime \right]\hae(s) - \left[\hu(s) \cdot \nabla^\prime\right]\hu(s), \varphi \right\rangle ds\right|\,dt\right] = 0$,
\item[(e)] $\lim_{\eps \to 0} \hE \left[\int_0^T \left|\int_0^t \left\langle\ocirc{M}_\eps \tfe (s) - f(s), \varphi \right\rangle ds\right|\,dt\right] = 0$,
\item[(f)] $\lim_{\eps \to 0} \hE \left[\int_0^T\left|  \left( \int_0^t \left( \oMe\left[\wtd{G}_\eps(s)\,d\what{W}_\eps(s)\right]  - G(s)\,d\what{W}(s)\right), \varphi\right)_\LS \right|^2 dt\right] = 0$.
\end{itemize}
\end{lemma}

\begin{proof}
Let us fix $\phi \in \rU$.

\noindent \textbf{(a)}
We know that $\hae \to \hu$ in $\calZ_T$. In particular,
\[
 \hae \rightarrow \what{u} \quad \text{ weakly in }\, C([0,T]; \rH)\quad \hp\text{-a.s.}.
\]
Hence, for $t \in [0,T]$
\begin{equation}
\label{eq:5.51}
\lim_{\eps \to 0} \left(\hae(t),\phi \right)_\LS = \left(\hu(t), \phi\right)_\LS\;, \quad \hp\text{-a.s.}.
\end{equation}
Since by \eqref{eq:5.46}, for every $\eps > 0$, $\sup_{t\in[0,T]}\|\hae(s)\|^2_\LS \le K_1(2)$, $\hp$-a.s., using the dominated convergence theorem we infer that
\begin{equation}
\label{eq:convg1}
\lim_{\eps \to 0} \int_0^T \left|\left(\hae(t) - \hu(t), \phi\right)_\LS\right| dt = 0\,,\quad \hp\text{-a.s.}
\end{equation}
By the H\"older inequality, \eqref{eq:5.46} and \eqref{eq:5.49} for every $\eps > 0$ and every $r > 1$
\begin{equation}
\label{eq:convg2}
\hE \left[\left|\int_0^T \|\hae(t) - \hu(t)\|_\LS dt\right|^r\right] \le c \hE \left[\int_0^T\left(\|\hae(t)\|_\LS^{r} + \|\hu(t)\|^{r}_\LS\right)dt \right] \le cTK_1(r),
\end{equation}
where $c$ is some positive constant. To conclude the proof of assertion $(a)$ it is sufficient to use \eqref{eq:convg1}, \eqref{eq:convg2} and the Vitali's convergence theorem.

\noindent \textbf{(b)} Since $\hae \to \hu$ in $C([0,T]; \rH_w)$ $\hp$-a.s. we infer that
\begin{equation}
\label{test-01}
    \left(\hae(0),\phi \right)_\LS \to \left(\hu(0),\phi \right)_\LS\,, \quad \hp\text{-a.s.}.
\end{equation}
Also, note that by condition~\eqref{eq:initial_data_convg} in Assumption~\ref{ass_sphere}, ${\alpha}_\eps(0) = \oMe \tue(0)=\oMe \widetilde{u}_0^\eps$ converges weakly to $u_0$ in $\mathbb{L}^2(\bbS)$.\\
On the other hand, by \eqref{eq:5.44} we infer that the laws of $\hae(0)$ and $\alpha_\eps(0)$ on $\rH$ are equal. Since $\ale(0)$ is a constant random variable on the old probability space, we infer that $\hae(0)$ is also a constant random variable (on the new probability space) and hence, by \eqref{eq:5.4} and \eqref{eq:notation_1}, we infer that $\hae(0) = \oMe \tu_0^\eps$ almost surely (on the new probability space). Therefore we infer that
\[\left(\what{u}(0), \varphi \right)_{\LS} = \left(u_0, \varphi \right)_{\LS},\]
concluding the proof of assertion $(b)$.

\noindent \textbf{(c)} Since $\hae \to \hu$ in $C([0,T]; \rH_w)$ $\hp$-a.s.,
\begin{align}
\label{eq:5.52}
& \lim_{\eps \to 0}{\nu}\int_0^t \left(\nablaS\hae(s), \nablaS\varphi \right)_\LS\,ds  = - \lim_{\eps \to 0} \nu \int_0^t \left(\hae (s), \Delta' \varphi \right)_\LS\,ds \\
&
\quad= - {\nu}\int_0^t \left(\hu(s), \Delta' \varphi \right)_\LS\,ds
 = \nu \int_0^t \left( \nabla^\prime \what{u}(s), \nabla^\prime \varphi \right)_\LS\,ds. \nn
\end{align}
The Cauchy-Schwarz inequality and estimate \eqref{eq:5.47} infer that for all $t \in [0,T]$ and $\eps \in (0,1]$
\begin{align}
\label{eq:convg3}
\hE & \left[\left|\int_0^t \nu \left(\nablaS \hae(s), \nablaS\varphi\right)_\LS\,ds\right|^2\right]\\
& \quad \le \nu^2 \|\nabla^\prime \varphi\|_\LS^2 \hE \left[\int_0^t \|\nabla^\prime \hae(s)\|^2_\LS\,ds\right] \le c K_2 \nn
\end{align}
for some constant $c > 0$. By \eqref{eq:5.52}, \eqref{eq:convg3} and the Vitali's convergence theorem we conclude that for all $t \in [0,T]$
\[
\lim_{\eps \to 0} \hE\left[ \left| \int_0^t \nu \left(\nablaS \hae(s) - \nablaS \hu(s),  \varphi \right)_\LS\,ds\right|\right] = 0.
\]
Assertion $(c)$ follows now from \eqref{eq:5.47}, \eqref{eq:5.48} and the dominated convergence theorem.

\noindent \textbf{(d)} For the non-linear term using the Sobolev embedding $\hH^2(Q) \hookrightarrow \lL^\infty(Q)$, we have
\begin{align}
\label{eq:5.53}
&\left| \int_0^{t} \left\langle\left[\hae(s)  \cdot \nablaS \right]\hae(s), \varphi \right)_\LS\,ds - \int_0^t \left\langle\left[\widehat{u}(s) \cdot \nablaS \right]\widehat{u}(s), \varphi \right\rangle\,ds\right| \\
&\le \left|\int_0^t \int_\bbS \left[\left(\hae(s,x) - \widehat{u}(s,x) \right) \cdot \nablaS \widehat{u}(s,x) \right]\cdot \varphi(x) \,dx\,ds\right| \nn \\
&\quad + \left|\int_0^t \int_\bbS \left[\hae(s,x) \cdot \nablaS \left(\hae(s,x) - \widehat{u}(s,x)\right)\right] \cdot \varphi(x) \,dx\,ds\right| \nn \\
& \le \|\hae - \widehat{u}\|_{L^2(0,T; \rH)}\|\widehat{u}\|_{L^2(0,T; \rV)}\|\varphi\|_{\hH^2(\bbS)} \nn \\
& \quad + \left|\int_0^t \left(\nablaS \left(\hae(s,x) - \widehat{u}(s,x)\right), \hae(s)\varphi\right)_\LS ds\right|. \nn
\end{align}
The first term converges to zero as $\eps \to 0$, since $\hae \to \hu$ strongly in $L^2(0,T; \rH)$ $\hp$-a.s., $\widehat{u} \in L^2(0,T; \rV)$ and the second term converges to zero too as $\eps \to 0$ because $\hae \to \widehat{u}$ weakly in $L^2(0,T; \rV)$. Using the H\"older inequality, estimates \eqref{eq:5.46} and the embedding $\hH^2(\bbS) \hookrightarrow \mathbb{L}^\infty(\bbS)$ we infer that for all $t \in [0,T]$, $\eps \in (0,1]$, the following inequalities hold
\begin{align}
\label{eq:convg4}
\hE & \left[\left|\int_0^t \left\langle\left[\hae(s)\cdot \nablaS\right]\hae(s),\varphi\right\rangle ds\right|^2\right] \nn \le \hE \left[\left(\int_0^t\|\hae(s)\|_\LS^2 \| \nablaS \varphi\|_{\mathbb{L}^\infty(\bbS)}\,ds\right)^2\right]\\
& \le c \|\nablaS \varphi\|_{\hH^2(\bbS)}\,t\, \hE \left[\int_0^t \|\hae(s)\|^4_\LS \,ds\right] \nn\\
& \le c \|\varphi\|_{\hH^3(\bbS)}\, t\,  \hE \left[\sup_{s \in [0,t]}\|\hae(s)\|^4_\LS \right] \le \widetilde{c} K_1(4).
\end{align}
By \eqref{eq:5.53}, \eqref{eq:convg4} and the Vitali's convergence theorem we obtain for all $t \in [0,T]$,
\begin{equation}
\label{eq:convg5}
\lim_{\eps \to 0} \hE \left[\left|\int_0^t \left\langle \left[\hae(s)\cdot\nablaS \right]\hae(s) - \left[\hu(s)\cdot \nablaS \right]\hu(s), \varphi \right\rangle\,ds\right|\right] = 0.
\end{equation}
Using the H\"older inequality and estimates \eqref{eq:5.46}, \eqref{eq:5.49}, we obtain for all $t \in [0,T]$, $\eps \in (0,1]$
\begin{align*}\hE \left[\left|\int_0^t \frac{1}{1+\eps}\left\langle\left[\hae(s)\cdot \nablaS\right] \hae(s) - \left[\hu(s) \cdot \nablaS \right]\hu(s), \varphi\right\rangle\,ds\right|\right] \nn \\
 \le c \|\varphi\|_{\hH^3(\bbS)}\hE \left[\sup_{t \in [0,T]}\|\hae(s)\|^2_\LS + \sup_{t \in[0,T]}\|\hu(s)\|^2_\LS\right] \le 2\widetilde{c} K_1(2)
\end{align*}
where $c, \widetilde{c} > 0$ are constants. Hence by \eqref{eq:convg5} and the dominated convergence theorem, we infer assertion $(d)$.

\noindent \textbf{(e)} Assertion $(e)$ follows because by Assumption~\ref{ass_sphere} the sequence $\left(\oMe \tfe\right)_{\eps \in(0,1]}$ converges in the sense of \eqref{eq:convg_ext_force} for every $\varphi \in \Psi$.

\noindent \textbf{(f)} By the definition of maps $\wtd{G}_\eps$ and $G$, we have
\begin{align*}
\int_0^t\left\|\left(\oMe\left[\widetilde{G}_\eps(s)\right] - G(s), \varphi \right)_\LS\right\|_{\mathcal{T}_2(\R^N; \R)}ds = \int_0^t \sum_{j=1}^N\left|\left(\oMe\tge^j(s) - g^j(s), \varphi\right)_\LS\right|^2\,ds.
\end{align*}
Since, by Assumption \ref{ass_sphere},  for every $j \in \{1, \cdots, N\}$, and $s \in [0,t]$, $\oMe \tge^j(s)$ converges weakly to $g^j(s)$ in $\LS$ as $\eps \to 0$, we get
\begin{equation}
\label{eq:5.54}
\lim_{\eps \to 0} \int_0^t \left\|\left(\oMe\left[\widetilde{G}_\eps(s)\right] - G(s), \varphi \right)_\LS\right\|^2_{\mathcal{T}_2(\R^N; \R)}ds = 0.
\end{equation}
By assumptions on $\tge^j$, we obtain the following inequalities for every $t \in [0,T]$ and $\eps \in (0,1]$
\begin{align}
\label{eq:5.55}
\hE & \left[\left|\int_0^t \left\|\left(\oMe\left[\widetilde{G}_\eps(s)\right] , \varphi \right)_\LS\right\|^2_{\mathcal{T}_2(\R^N;\R)}\,ds\right|^2\right] \nn\\
& \le c \|\varphi\|_\LS^4 \left[\int_0^t\|\oMe(\widetilde{G}_\eps(s))\|^4_{\mathcal{T}_2(\R^N;\LS)}\,ds\right] \nn \\
& = c \|\varphi\|^4_\LS \left[\int_0^t \dfrac{1}{\eps^2}\left(\sum_{j=1}^N \eps \|\oMe(\tge^j(s))\|^2_\LS\right)^2 ds\right] \nn \\
&  \le \dfrac{\widetilde{c}}{\eps^2}\|\varphi\|^4_\LS \left[\sum_{j=1}^N \int_0^t \|\tge^j(s)\|^4_\LQeps ds\right] \le K,
\end{align}
where $c, \widetilde{c} > 0$ are some constants. Using the Vitali's convergence theorem, by \eqref{eq:5.54} and \eqref{eq:5.55} we infer
\begin{equation}
\label{eq:5.56}
\lim_{\eps \to 0} \hE \left[ \int_0^t \left\|\left(\oMe\left[\widetilde{G}_\eps(s)\right] - G(s), \varphi \right)_\LS\right\|_{\mathcal{T}_2(\R^N; \R)}^2ds\right] = 0.
\end{equation}
Hence, by the properties of the It\^o integral we deduce that for all $t \in [0,T]$,
\begin{equation}
\label{eq:5.57}
\lim_{\eps \to 0} \hE \left[\left| \left( \int_0^t \left[\oMe \left[\widetilde{G}_\eps(s)\right] - G(s)\right] d \widehat{W}(s), \varphi \right)_\LS \right|^2 \right] = 0.
\end{equation}
By the It\^o isometry and assumptions on $\tge^j$ and $g^j$ we have for all $t \in [0,T]$ and $\eps \in (0,1]$
\begin{align}
\label{eq:5.58}
\hE & \left[\left|\left(\int_0^t \left[\oMe\left[\widetilde{G}_\eps(s)\right] - G(s)\right]d\widehat{W}(s) , \varphi \right)_\LS\right|^2\right] \nn \\
& = \hE \left[\int_0^t \left\|\left(\oMe \left[\widetilde{G}_\eps(s)\right] - G(s), \varphi \right)_\LS\right\|^2_{\mathcal{T}_2(\R^N; \R)}\,ds\right] \nn \\
& \le c \|\varphi\|^2_\LS \left[\sum_{j=1}^N \int_0^t  \left(\|\oMe(\tge^j(s))\|^2_\LS + \|g^j(s)\|^2_\LS\right)\,ds \right] \nn \\
& = c \|\varphi\|^2_\LS \left[\sum_{j=1}^N \int_0^t  \left(\frac{1}{\eps}\|\tMe \tge^j(s)\|^2_\LQeps + \|g^j(s)\|^2_\LS\right)\,ds \right] \nn \\
& \le c \|\varphi\|^2_\LS \left[\sum_{j=1}^N \int_0^t  \left(\frac{1}{\eps}\|\tge^j(s)\|^2_\LQeps + \|g^j(s)\|^2_\LS\right)\,ds \right] \le \widetilde{K},
\end{align}
where $c > 0$ is a constant. Thus, by \eqref{eq:5.57}, \eqref{eq:5.58} and the dominated convergence theorem assertion $(f)$ holds.
\end{proof}

\begin{lemma}
\label{lemma_convg_beta}
For all $t \in [0,T]$ and $\varphi \in \Psi$, we have (along the sequence $\eps_{k_n}$)
\[
\lim_{\eps \to 0} \hE \left[ \int_0^T \left| \frac{1}{\eps} \int_0^t \left\langle [\what{\beta}_\eps(s)\cdot \nabla]\what{\beta}_\eps(s), \test \right\rangle_\eps\,ds \right|dt\right] = 0.
\]
\end{lemma}

\begin{proof}
Let us fix $\varphi \in \Psi$. Using the H\"older inequality, inequality \eqref{eqn-NepsL3}, embedding $\hH^1(Q_\eps) \hookrightarrow \mathbb{L}^6(Q_\eps)$ and estimate \eqref{eq:beta_estimate_1}, we get for $t \in [0,T]$
\begin{align*}
\frac{1}{\eps}& \what{\E}  \left| \int_0^t \left\langle [\what{\beta}_\eps(s)\cdot \nabla]\what{\beta}_\eps(s), \test \right\rangle_\eps\,ds \right|
 = \frac{1}{\eps} \what{\E} \left| \int_0^t \left(\int_{Q_\eps} \left([\hbe(s,\bx)\cdot \nabla]\hbe(s,\bx)\right)\cdot [\test](\bx)\,d\bx \right)\,ds \right| \\
 &\le \frac{1}{\eps} \what{\E} \int_0^t \|\hbe(s)\|_{\mathbb{L}^3(Q_\eps)} \|\nabla \hbe(s)\|_\LQeps \|\test\|_{\mathbb{L}^6(Q_\eps)}\,ds\\
& \le \frac{1}{\eps} \what{\E} \left[\int_0^t \sqrt{c_0 \eps} \|\hbe(s)\|_{\rV_\eps} \|\nabla \hbe(s)\|_{\LQeps}\,ds \right]{\eps^{1/6}}\|\varphi\|_{\mathbb{L}^6(Q)} \\
& \le \eps^{-\frac13} \left(\what{\E} \int_0^T \|\hbe(s)\|_{\rV_\eps}^2\right)^{1/2} \left(\what{\E} \int_0^T \|\nabla \hbe(s)\|^2_{\LQeps}\,ds\right)^{1/2}\|\varphi\|_{\rV} \\
&\le \eps^{-\frac13} \sqrt{K_3(2) \eps} \sqrt{K_4 \eps} \|\varphi\|_{\rV} = \sqrt{K_3(2)}\sqrt{K_4} \eps^{\frac23} \|\varphi\|_{\rV}.
\end{align*}
Thus
\begin{equation}
\label{eq:convg_beta_3}
\lim_{\eps \to 0}\frac{1}{\eps}\hat{\E}  \left| \int_0^t \left( \left[\hbe(s) \cdot \nabla \right]\hbe(s), \test \right)_\LQeps\,ds \right| = 0\;.
\end{equation}
Hence we can conclude the proof of the lemma by using dominated convergence theorem, estimates \eqref{eq:beta_estimate_1}--\eqref{eq:beta_estimate_2} and convergence \eqref{eq:convg_beta_3}.
\end{proof}

Finally, to finish the proof of Theorem~\ref{thm:main_thm}, we will follow the methodology as in \cite{[Motyl14]} and introduce some auxiliary notations (along sequence $\eps_{k_n}$)
\begin{align}
\label{eq:5.59}
&\Lambda_\eps(\hae,\hbe,\widehat{W}_\eps, \varphi) :=  \left(\hae(0), \varphi \right)_\LS - \nu \int_0^t \left( \nablaS \hae(s), \nablaS \varphi \right)_\LS\,ds \\
&\qquad  - \int_0^{t} \left\langle \left[\hae(s) \cdot \nabla^\prime \right]\hae(s), \varphi \right\rangle\,ds + \int_0^{t} \left\langle \ocirc{M}_\eps \wtd{f}_\eps(s), \varphi \right\rangle\,ds \nn \\
&\qquad + \left(\int_0^t \ocirc{M}_\eps \left[\wtd{G}_\eps(s)\,d\what{W}_\eps(s) \right], \varphi \right)_\LS
- \frac{1}{\eps} \int_0^{t} \left\langle\left[\hbe(s) \cdot \nabla \right]\hbe(s), \test \right\rangle_\eps\,ds \nn\;,
\end{align}

\begin{align}
\label{eq:5.60}
&\Lambda(\hu,\widehat{W}, \varphi) :=  \left(\hu(0), \varphi \right)_\LS -\nu \int_0^t \left(\nablaS \hu(s), \nablaS \varphi \right)_\LS\,ds \\
&\qquad  - \int_0^{t} \left\langle\left[\hu(s) \cdot \nabla^\prime \right]\hu(s), \varphi \right\rangle_\LS\,ds + \int_0^{t} \left\langle f(s), \varphi \right\rangle\,ds \nn \\
&\quad + \left(\int_0^t G(s)\,d\what{W}(s), \varphi \right)_\LS \nn.
\end{align}

\begin{corollary}
\label{cor5.12}
Let $\varphi \in \Psi$. Then (along the sequence $\eps_{k_n}$)
\begin{equation}
\label{eq:5.61}
\lim_{\eps \to 0} \left\|\left(\hae(\cdot), \varphi \right)_\LS - \left(\hu(\cdot), \varphi\right)_\LS\right\|_{L^1(\what{\Omega} \times [0,T])} = 0
\end{equation}
and
\begin{equation}
\label{eq:5.62}
\lim_{\eps \to 0} \left\|\Lambda_\eps(\hae, \hbe, \widehat{W}_\eps, \varphi) - \Lambda(\hu, \widehat{W}, \varphi)\right\|_{L^1(\what{\Omega} \times [0,T])} = 0\;.
\end{equation}
\end{corollary}
\begin{proof}
Assertion \eqref{eq:5.61} follows from the equality
\[\left\|\left(\hae(\cdot), \varphi \right)_\LS - \left(\hu(\cdot), \varphi\right)_\LS\right\|_{L^1(\what{\Omega} \times [0,T])} = \hE \left[ \int_0^T \left|\left(\hae(t) - \hu(t), \varphi \right)_\LS\right|\,dt\right]\]
and Lemma~\ref{lemma5.8} $(a)$. To prove assertion \eqref{eq:5.62}, note that by the Fubini Theorem, we have
\[
\left\|\Lambda_\eps(\hae, \hbe, \widehat{W}_\eps, \varphi) - \Lambda(\hu, \widehat{W}, \varphi)\right\|_{L^1(\what{\Omega} \times [0,T])} =  \int_0^T \hE \left|\Lambda_\eps(\hae, \hbe, \widehat{W}_\eps, \varphi)(t) - \Lambda(\hu, \widehat{W}, \varphi)(t)\right| dt.
\]
To conclude the proof of the corollary, it is sufficient to note that by Lemma~\ref{lemma5.8} $(b)-(f)$ and Lemma~\ref{lemma_convg_beta}, each term on the right hand side of \eqref{eq:5.59} tends at least in $L^1(\what{\Omega} \times [0,T])$ to the corresponding term (to zero in certain cases) in \eqref{eq:5.60}.
\end{proof}

\begin{proof}[Conclusion of proof of Theorem~\ref{thm:main_thm}] Let us fix $\varphi \in \Psi$. Since $\ale$ is a solution of \eqref{eq:snse_mod} for all $t \in [0,T]$
\[\left(\ale(t), \varphi \right)_\LS = \Lambda_\eps(\ale, \tble, \widetilde{W}_\eps, \varphi)(t)\,, \qquad \mathbb{P}\text{-a.s.}\]
In particular,
\[
\int_0^T \E \left|\left(\ale(t),\varphi\right)_\LS - \Lambda_\eps(\ale, \tble, \widetilde{W}, \varphi)(t)\right|\,dt = 0.
\]
Since $\mathcal{L}(\ale, \eta_{k_n}, \widetilde{W}_\eps) = \mathcal{L}(\hae, \what{\eta}_{k_n}, \widehat{W}_\eps)$ on $\mathcal{B}\left(\calZ_T \times \mathcal{Y}_T \times C([0,T]; \R^N)\right)$ (along the sequence $\eps_{k_n}$),
\[
\int_0^T \hE \left|\left(\hae(t),\varphi\right)_\LS - \Lambda_\eps(\hae, \hbe, \widehat{W}_\eps, \varphi)(t)\right|\,dt = 0.
\]
Therefore by Corollary~\ref{cor5.12} and the definition of $\Lambda$, for almost all $t \in [0,T]$ and $\widehat{\mathbb{P}}$-almost all $\omega \in \what{\Omega}$
\[
\left(\hu(t), \varphi \right)_\LS - \Lambda(\hu, \widehat{W}, \varphi)(t) = 0,
\]
i.e. for almost all $t \in [0,T]$ and $\widehat{\mathbb{P}}$-almost all $\omega \in \what{\Omega}$
\begin{align}
\label{eq:5.63}
\left(\hu(t), \varphi\right)_\LS + \nu \int_0^t \left(\nablaS\hu(s), \nablaS\varphi \right)_\LS\,ds + \int_0^{t} \left\langle\left[\hu(s) \cdot \nabla^\prime \right]\hu(s), \varphi \right\rangle\,ds  \\
 = \left(\hu(0), \varphi \right)_\LS + \int_0^{t} \left\langle f(s), \varphi \right\rangle\,ds + \left(\int_0^t G(s)\,d\widehat{W}(s), \varphi \right)_\LS \nn\;.
\end{align}

Putting $\widehat{\mathcal{U}} := \left(\what{\Omega}, \what{\mathcal{F}}, \what{\bbF}, \hp\right)$, we infer that the system $\left(\widehat{\mathcal{U}}, \widehat{W}, \hu\right)$ is a martingale solution to \eqref{eq:5.6}--\eqref{eq:5.9}.
\end{proof}

\appendix
\section{Compactness}
\label{sec:appA}

\subsection{Skorohod Theorem and Aldous condition}
\label{sec:appA.1}

Let $E$ be a separable Banach space with the norm $|\cdot|_E$ and let $\calB(E)$ be its Borel $\sigma$-field. The family of probability measures on $(E, \calB(E))$ will be denoted by $\Lambda$. The set of all bounded and continuous $E$-valued functions is denoted by $C_b(E)$.

\begin{definition}
\label{defnA.1.1}
The family $\Lambda$ of probability measures on $\left(E, \calB(E)\right)$ is said to be tight if for arbitrary $\varepsilon > 0$ there exists a compact set $K_\varepsilon \subset E$ such that
\[\mu(K_\varepsilon) \ge 1 - \varepsilon\,, \quad \mbox{for all}\,\, \mu \in \Lambda\,.\]
\end{definition}

We will need the following Jakubowski's generalisation of the Skorohod Theorem, in the form given by Brze\'{z}niak and Ondrej\'{a}t \cite[Theorem~C.1]{[BO11]}, see also \cite{[Jakubowski97]}, as we deal with non-metric spaces.

\begin{theorem}
\label{thmA.1.4}
Let $\mathcal{X}$ be a topological space such that there exists a sequence $\{f_m\}_{m \in \mathbb{N}}$ of continuous functions $f_m : \mathcal{X} \to \R$ that separates points of $\mathcal{X}$. Let us denote by $\mathcal{S}$ the $\sigma$-algebra generated by the maps $\{f_m\}$. Then
\begin{itemize}
\item[a)] every compact subset of $\mathcal{X}$ is metrizable,
\item[b)] if $(\mu_m)_{m \in \mathbb{N}}$ is a tight sequence of probability measures on $(\mathcal{X}, \mathcal{S})$, then there exists a subsequence $(m_k)_{k \in \mathbb{N}}$, a probability space $(\Omega, \mathcal{F}, \mathbb{P})$ with $\mathcal{X}$-valued Borel measurable variables $\xi_k, \xi$ such that $\mu_{m_k}$ is the law of $\xi_k$ and $\xi_k$ converges to $\xi$ almost surely on $\Omega$.
\end{itemize}
\end{theorem}

Let $(\mathbb{S}, \varrho)$ be a separable and complete metric space.

\begin{definition}
\label{defnA.1.5}
Let $u \in C([0,T]; \mathbb{S})$. The modulus of continuity of $u$ on $[0,T]$ is defined by
\[m(u, \delta) :=  \sup_{s,t \in [0,T],\,|t - s|\le \delta} \varrho (u(t), u(s)), \quad \delta > 0\,.\]
\end{definition}

Let $(\Omega, \mathcal{F}, \mathbb{P})$ be a probability space with filtration $\mathbb{F}:= (\mathcal{F}_t)_{t \in [0,T]}$ satisfying the usual conditions, see \cite{[Metivier82]},
and let $(X_n)_{n \in \mathbb{N}}$ be a sequence of continuous $\mathbb{F}$-adapted $\mathbb{S}$-valued processes.

\begin{definition}
\label{defnA.1.6}
We say that the sequence $(X_n)_{n \in \mathbb{N}}$ of $\mathbb{S}$-valued random variables satisfies condition $[\mathbf{T}]$ iff $\forall\, \varepsilon >0, \forall\, \eta > 0,\, \exists\, \delta > 0$:
\begin{equation}
\label{eq:A.1.1}
\sup_{n \in \mathbb{N}} \mathbb{P}\left\{m(X_n, \delta) > \eta\right\} \le \varepsilon\,.
\end{equation}
\end{definition}

\begin{lemma}\cite[Lemma~2.4]{[BM14]}
\label{lemmaA.1.7}
Assume that $(X_n)_{n \in \mathbb{N}}$ satisfies condition $[\mathbf{T}]$. Let $\mathbb{P}_n$ be the law of $X_n$ on $C([0,T]; \mathbb{S})$, $n \in \mathbb{N}$. Then for every $\varepsilon > 0$ there exists a subset $A_\varepsilon \subset C([0,T]; \mathbb{S})$ such that
\[\sup_{n \in \mathbb{N}} \mathbb{P}_n(A_\varepsilon) \ge 1 - \varepsilon\]
and
\begin{equation}
\label{eq:A.1.2}
\lim_{\delta \to 0} \sup_{u \in A_\varepsilon} m(u, \delta) = 0\,.
\end{equation}
\end{lemma}

Now we recall the Aldous condition $[\mathbf{A}]$, which is connected with condition $[\mathbf{T}]$ (see \cite{[Metivier88]} and \cite{[Aldous78]}). This condition allows to investigate the modulus of continuity for the sequence of stochastic processes by means of stopped processes.

\begin{definition}[Aldous condition]
\label{defnA.1.8}
A sequence $(X_n)_{n \in \mathbb{N}}$ satisfies condition $[\mathbf{A}]$ iff $\,\forall\, \varepsilon > 0$, $\forall\, \eta > 0$, $\exists \, \delta > 0$ such that for every sequence $(\tau_n)_{n \in \mathbb{N}}$ of $\mathbb{F}$-stopping times with $\tau_n \le T$ one has
\[\sup_{n \in \mathbb{N}} \sup_{0 \le \theta \le \delta} \mathbb{P}\left\{\varrho(X_n(\tau_n + \theta), X_n(\tau_n)) \ge \eta \right\} \le \varepsilon\,.\]
\end{definition}

\begin{lemma}
\label{lemmaA.1.9}
\cite[Theorem~3.2]{[Metivier88]}
Conditions $[\mathbf{A}]$ and $[\mathbf{T}]$ are equivalent.
\end{lemma}

\subsection{Tightness criterion}
\label{sec:appA.2}
Now we formulate the compactness criterion analogous to the result due to Mikulevicus and Rozowskii \cite{[MR05]}, Brze\'zniak and Motyl \cite{[BM14]} for the space $\mathcal{Z}_T$, see also \cite[Lemma~4.2]{[BD18]}.

\begin{lemma}
\label{lemmaA.2.1}
Let $\mathcal{Z}_T$, $\mathcal{T}$ be as defined in \eqref{def:ZT}. Then a set $\calK \subset \mathcal{Z}_T$ is $\mathcal{T}$-relatively compact if the following three conditions hold
\begin{itemize}
\item[(a)] $\sup_{u \in \calK} \sup_{s \in [0,T]} \|u(s)\|_\LS < \infty\,,$
\item[(b)] $\sup_{u \in \calK} \int_0^T \|u(s)\|^2_\rV\,ds < \infty\,$, i.e. $\calK$ is bounded in $L^2(0,T; \rV)$,
\item[(c)] $\lim_{\delta \to 0} \sup_{u \in \calK} \sup_{\underset{|t-s| \le \delta}{s,t \in [0,T]}}\|u(t) - u(s)\|_{\rU^\ast} = 0\,.$
\end{itemize}
\end{lemma}

Using \S\,\ref{sec:appA.1} and the compactness criterion from Lemma~\ref{lemmaA.2.1} we obtain the following corollary.

\begin{corollary}[Tightness criterion]
\label{corA.2.2}
Let $(X_n)_{n \in \mathbb{N}}$ be a sequence of continuous $\mathbb{F}$-adapted $\mathrm{H}$-valued processes such that
\begin{itemize}
\item[(a)] there exists a constant $C_1 > 0$ such that
\[\inf_{\eps > 0}\,\E \left[ \sup_{s \in [0,T]} \|\ale(s)\|^2_{\LS} \right] \le C_1,\]

\item[(b)] there exists a constant $C_2 > 0$ such that
\[\inf_{\eps > 0}\,\E \left[ \int_0^T \|\nabla^\prime \ale(s)\|^2_{\LS}\,ds \right] \le C_2,\]

\item[(c)] $(\ale)_{\eps > 0}$ satisfies the Aldous condition $[\mathbf{A}]$ in $\rU^\ast$.
\end{itemize}
Let $\wtd{\mathbb{P}}_\eps$ be the law of $\ale$ on $\mathcal{Z}_T$. Then for every $\delta > 0$ there exists a compact subset $K_\delta$ of $\mathcal{Z}_T$ such that
\[ \inf_{\eps > 0} \tilde{\mathbb{P}}_\eps(K_\delta) \ge 1 - \delta.\]
\end{corollary}

\section{Kuratowski Theorem and proof of Lemma~\ref{lemma5.7}}
\label{s:AppB}
This appendix is dedicated to the proof of Lemma~\ref{lemma5.7}. We will first recall the Kuratowski Theorem \cite{[Kuratowski52]} in the next subsection and prove some related results which will be used later to prove Lemma~\ref{lemma5.7} in \S\,\ref{s:AppB.2}.

\subsection{Kuratowski Theorem and related results}
\label{s:AppB.1}

\begin{theorem}
\label{thmB.1.1}
Assume that $X_1, X_2$ are the Polish spaces with their Borel $\sigma$-fields denoted respectively by $\calB(X_1), \calB(X_2)$. If $\varphi \colon X_1 \to X_2$ is an injective Borel measurable map then for any $E_1 \in \calB(X_1)$, $E_2 := \varphi(E_1) \in \calB(X_2)$.
\end{theorem}

Next two lemmas are the main results of this appendix. For the proof of Lemma~\ref{lemmaB.1.2} see \cite[Appendix~B]{[BD19]}.

\begin{lemma}
\label{lemmaB.1.2}
Let $X_1, X_2$ and $Z$ be topological spaces such that $X_1$ is a Borel subset of $X_2$. Then $X_1 \cap Z$ is a Borel subset of $X_2 \cap Z$, where $X_2 \cap Z$ is a topological space too, with the topology given by
\begin{equation}\label{eq:B.1.1}
\tau(X_2 \cap Z) = \left\{ A \cap B : A \in \tau(X_2), B \in \tau(Z)\right\}.
\end{equation}
\end{lemma}

\subsection{Proof of Lemma~\ref{lemma5.7}}
\label{s:AppB.2}

In this subsection we recall Lemma~\ref{lemma5.7} and prove it using the results from previous subsection.

\begin{lemma}
\label{lemmaB.2.1}[Also, Lemma~\ref{lemma5.7}]
Let $T > 0$ and $\mathcal{Z}_T$ be as defined in \eqref{def:ZT}. Then, the following sets $C([0,T];\rH) \cap \mathcal{Z}_T$, $L^2(0,T; \rV) \cap \mathcal{Z}_T$ are Borel subsets of $\mathcal{Z}_T$.
\end{lemma}

\begin{proof}
First of all $C([0,T]; \rH) \subset C([0,T]; \rU^\ast) \cap L^2(0,T; \rH)$. Secondly, $C([0,T]; \rH)$ and $C([0,T]; \rU^\ast) \cap L^2(0,T; \rH)$ are Polish spaces. And finally, since $\rH$ is continuously embedded in $\rU^\ast$, the map
\[i \colon C([0,T]; \rH) \to C([0,T]; \rU^\ast) \cap L^2(0,T; \rH),\]
is continuous and hence Borel. Thus, by application of the Kuratowski Theorem (see Theorem~\ref{thmB.1.1}) $C([0,T]; \rH)$ is a Borel subset of $C([0,T]; \rU^\ast) \cap L^2(0,T; \rH)$. Therefore, by Lemma~\ref{lemmaB.1.2}, $C([0,T]; \rH) \cap {\mathcal{Z}}_T$ is a Borel subset of $C([0,T]; \rU^\ast) \cap L^2(0,T; \rH) \cap {\mathcal{Z}}_T$ which is equal to ${\mathcal{Z}}_T$.

Similarly we can show that $L^2(0,T; \rV) \cap {\mathcal{Z}}_T$ is a Borel subset of ${\mathcal{Z}}_T$. $L^2(0,T; \rV) \hookrightarrow L^2(0,T; \rH)$ and both are Polish spaces thus by application of the Kuratowski Theorem, $L^2(0,T; \rV)$ is a Borel subset of $L^2(0,T; \rH)$. Finally, we can conclude the proof of lemma by Lemma~\ref{lemmaB.1.2}.
\end{proof}


\end{document}